\documentclass[10 pt, leqno]{article}
\date{}
\usepackage{amssymb,amsbsy,amsmath,amsfonts,amssymb,amscd, mathrsfs, amsthm}
\usepackage[english]{babel}
\usepackage[T1]{fontenc}
\usepackage{indentfirst}
\usepackage{color}

\makeatletter
\@addtoreset{equation}{section}
\makeatother

\newtheorem{statement}{}[section]
\newtheorem{theorem}[statement]{Theorem}
\newtheorem{lemma}[statement]{Lemma}

\newtheorem{proposition}[statement]{Proposition}

\newcommand\C{\mathbb C}

\newcommand\T{\mathbb T}
\newcommand\D{\mathbb D}

\newcommand\e{{\rm e}}

\newcommand\eps{\varepsilon}
\newcommand\ind{{\rm 1\kern-.30em I}}
\newcommand\dis{\displaystyle}
\renewcommand \Re{{\mathfrak R}{\rm e}\,}

\newcommand\converge{\mathop{\longrightarrow}\limits}
\newcommand\capa{{\rm Cap}\,}
\newcommand\capam{{\rm Cap}_m\,}

\title{\bf Some examples of composition operators and their approximation numbers on the Hardy space of the bi-disk}
\author{\it Daniel~Li,  Herv\'e~Queff\'elec, L.~Rodr{\'\i}guez-Piazza}

\date{\footnotesize \today}

\begin{document}

\maketitle

\noindent {\bf Abstract.} We give examples of composition operators $C_\Phi$ on $H^2 (\D^2)$ showing that the condition $\|\Phi \|_\infty = 1$ is 
not sufficient for their approximation numbers $a_n (C_\Phi)$ to satisfy $\lim_{n \to \infty} [a_n (C_\Phi) ]^{1/\sqrt{n}} = 1$, contrary to the 
$1$-dimensional case. We also give a situation where this implication holds. We make a link with the Monge-Amp\`ere capacity of the image of $\Phi$.
\medskip

\noindent \emph{Key-words}\/: approximation numbers; Bergman space; bidisk; composition operator; Green capacity; Hardy space; 
Monge-Amp\`ere capacity; weighted composition operator.
\medskip

\noindent \emph{MSC~2010 numbers} -- \emph{Primary}\/: 47B33 -- \emph{Secondary}\/: 30H10 -- 30H20 -- 31B15 -- 32A35 -- 32U20 -- 41A35 -- 46B28


\section{Introduction and notation} \label{sec: intro}

\subsection{Introduction}

The purpose of this paper is to continue the study of composition operators on the polydisk initiated in \cite{BLQR}, and in particular to examine to what 
extent one of the main results of \cite{LQR} still holds. 
\medskip

Let $H$ be a Hilbert space and $T \colon H\to H$ a bounded operator. Recall that the \emph{approximation numbers} of $T$ are defined as:
\begin{displaymath} 
\qquad \qquad a_n (T) = \inf_{{\rm rank}\, R < n} \| T - R\| \, , \quad n \geq 1 \, , 
\end{displaymath} 
and we have: 
\begin{displaymath} 
\| T \| = a_1 (T) \geq a_2 (T) \geq \cdots \geq a_n (T) \geq \cdots
\end{displaymath} 
 The operator $T$ is compact if and only if $a_n (T) \converge_{n \to \infty} 0$. \par
\smallskip

For $d \geq 1$, we define:
\begin{displaymath}
\raise 4 pt \hbox{$\Bigg\{ $}
\begin{array}{rl}
\beta_{d}^{-}(T) & \dis = \liminf_{n\to \infty}\big[a_{n^d}(T)\big]^{1/n} \\ 
\medskip
\beta_{d}^{+}(T) & \dis = \limsup_{n\to \infty}\big[a_{n^d}(T)\big]^{1/n} 
\end{array}
\end{displaymath}
We have:
\begin{displaymath} 
0\leq \beta_{d}^{-}(T)\leq \beta_{d}^{+}(T) \leq 1 \, ,
\end{displaymath} 
and we simply write $\beta_{d}(T)$ in case of equality. 
\smallskip

It may well happen in general (consider diagonal operators) that $\beta_{d}^{-}(T) = 0$ and  $\beta_{d}^{+}(T)=1$.\par
\medskip

When $H = H^2 (\D)$ is the Hardy space on the open unit disk $\D$ of $\C$, and $T = C_\Phi$ is a composition operator, with 
$\Phi \colon \D \to \D$ a non-constant analytic function, we always have (\cite{LIQUEROD}):
\begin{displaymath} 
\beta_{1}^{-}(C_\Phi) > 0 \, ,
\end{displaymath} 
and one of the main results of \cite{LIQUEROD} is the equivalence:
\begin{equation} \label{equiv-dim1}
\beta_{1}^+ (C_\Phi) < 1 \quad \Longleftrightarrow \quad \Vert \Phi\Vert_\infty < 1 \, .
\end{equation} 
An alternative proof was given in \cite{LQR}, as a consequence of a so-called ``spectral radius formula'', which moreover shows that:
\begin{displaymath} 
\beta_{1}^- (C_\Phi) = \beta_{1}^+ (C_\Phi) \, .
\end{displaymath} 

In \cite{BLQR}, for $d \geq 2$, it is proved that, for a bounded symmetric domain $\Omega \subseteq \C^d$, if  
$\Phi \colon \Omega \to \Omega$ is analytic, such that $\Phi (\Omega)$ has a non-void interior, and the composition  operator 
$C_\Phi \colon H^2 (\Omega) \to H^2 (\Omega)$ is compact, then:
\begin{displaymath} 
\beta_d^- (C_\Phi) > 0 \, .
\end{displaymath} 
On the other hand, if $\Omega$ is a product of balls, then:
\begin{displaymath} 
\| \Phi \|_\infty < 1 \quad \Longrightarrow \quad \beta_d^+ (C_\Phi) < 1 \, . 
\end{displaymath} 
We do not know whether the converse holds and the purpose of this paper is to study some examples towards an answer. 
\bigskip

The paper is organized as follows. Section~\ref{sec: intro} is this short introduction, as well as some notations and definitions on singular numbers of operators 
and Hardy spaces of the polydisk to follow. Section~\ref{sec: weighted} contains preliminary results on weighted composition operators in one variable, which 
surprisingly play an important role in the study of non-weighted composition operators in two variables. Section~\ref{sec: splitted} studies the case of symbols 
with ``separated'' variables. Our main one variable result extends in this case. Section~\ref{sec: glued} studies the ``glued case'' 
$\Phi (z_1, z_2) = \big( \phi (z_1), \phi (z_1) \big)$ for which even boundedness is an issue. Here, the Bergman space $B^{2}(\D)$ enters the picture. 
Section~\ref{sec: triangularly} studies the case of ``triangularly separated'' variables. This section lets direct Hilbertian sums of weighted composition operators in 
one variable appear, and it contains our main result: an example of a symbol  $\Phi$ satisfying $\Vert \Phi \Vert_\infty = 1$ and yet 
$\beta_{2}^{+} (C_\Phi) < 1$. The final Section~\ref{sec: capacity} discusses the role of the Monge-Amp\`ere pluricapacity, which is a multivariate extension 
of the Green capacity in the disk. Even though, as evidenced by our counterexample of Section~\ref{sec: triangularly}, this capacity will not capture all the 
behavior of the parameter $\beta_{m} (C_\Phi)$, some partial results are obtained, relying on theorems of S.~Nivoche and V.~Zakharyuta.

\subsection{Notation} 

We denote by $\D$ the open unit disk of the complex plane and by $\T$ its boundary, the $1$-dimensional torus. 

The Hardy space $H^2 (\D^d)$ is the space of holomorphic functions $f \colon \D^d \to \C$ whose boundary values $f^\ast$ on $\T^d$ are square-integrable 
with respect to the Haar measure $m_d$ of $\T^d$, and normed with:
\begin{displaymath} 
\| f \|_2^2 = \| f \|_{H^2 (\D^d)}^2 = \int_{\T^d} |f^\ast (\xi_1, \ldots, \xi_d) |^2 \, dm_d (\xi_1, \ldots, \xi_d) \, .
\end{displaymath} 
If $f (z_1, \ldots, z_d) = \sum_{\alpha_1, \ldots, \alpha_d \geq 0} a_{\alpha_1, \ldots, \alpha_d} \, z_1^{\alpha_1} \cdots z_d^{\alpha_d}$, then:
\begin{displaymath} 
\| f \|_2^2 = \sum_{\alpha_1, \ldots, \alpha_d \geq 0} | a_{\alpha_1, \ldots, \alpha_d} |^2 \, .
\end{displaymath} 

We say that an analytic map $\Phi \colon \D^d \to \D^d$ is a \emph{symbol} if its associated composition operator 
$C_\Phi \colon H^2 (\D^d) \to H^2 (\D^d)$, defined by $C_\Phi (f) = f \circ \Phi$, is bounded. 

We say that $\Phi$ is \emph{truly $d$-dimensional} if $\Phi (\D^d)$ has a non-void interior. 
\smallskip

We will make use of two kinds of symbols defined on $\D$.
\smallskip

The \emph{lens map} $\lambda_\theta \colon \D \to \D$ is defined, for $\theta \in (0, 1)$, by:
\begin{equation} 
\lambda_\theta (z) = \frac{(1 + z)^\theta - (1 - z)^\theta}{(1 + z)^\theta + (1 - z)^\theta} 
\end{equation} 
(see \cite{Shapiro-livre}, p.~27, or \cite{LELIQR}, for more information), and corresponds to $u \mapsto u^\theta$ in the right half-plane. 
\smallskip

The \emph{cusp map} $\chi \colon \D \to \D$ was first defined in \cite{LELIQR-TAMS} and in a slightly different form in \cite{LIQURO-Estimates}; we 
actually use here the modified form introduced in \cite{LELIQR-capacity}, and then used in \cite{LELIQR-approx-Dirichlet}. We first define:
\begin{displaymath}
\chi_0 (z) = \frac{\displaystyle \Big( \frac{z - i}{i z - 1} \Big)^{1/2} - i} {\displaystyle - i \, \Big( \frac{z - i}{i z - 1} \Big)^{1/2} + 1} \, ; 
\end{displaymath}
we note that $\chi_0 (1) = 0$, $\chi_0 (- 1) = 1$, $\chi_0 (i) = - i$, $\chi_0 (- i) = i$, and $\chi_0 (0) = \sqrt{2} - 1$. Then we set: 
\begin{displaymath}
\chi_1 (z) = \log \chi_0 (z), \quad \chi_2 (z) = - \frac{2}{\pi}\, \chi_1 (z) + 1, \quad \chi_3 (z) = \frac{a}{\chi_2 (z)}  \, \raise 1pt \hbox{,} 
\end{displaymath}
and finally:
\begin{displaymath}
\chi (z) = 1 - \chi_3 (z) \, ,
\end{displaymath}
where:
\begin{equation} \label{definition de a}
a = 1 - \frac{2}{\pi} \log (\sqrt{2} - 1) \in (1, 2) 
\end{equation} 
is chosen in order that $\chi (0) = 0$. The image $\Omega$ of the (univalent) cusp map is formed by the intersection of the  inside of the disk 
$D \big(1 - \frac{a}{2} \raise 1pt \hbox{,} \frac{a}{2} \big)$ and the outside of the two disks  
$D \big(1 + \frac{i a}{2} \raise 1pt \hbox{,} \frac{a}{2} \big)$ and $D \big(1  - \frac{ i a}{2} \raise 1pt \hbox{,} \frac{a}{2} \big) \cdot$ 
\medskip

Besides the approximation numbers, we need other singular numbers for an operator $S \colon X \to Y$ between Banach spaces $X$ and $Y$. \par
\smallskip

The \emph{Bernstein numbers} $b_n (S)$, $n \geq 1$, which are defined by:
\begin{equation} \label{Bernstein}
b_n (S) = \sup_E \min_{x \in S_E} \| S x\| \, ,
\end{equation} 
where the supremum is taken over all $n$-dimensional subspaces of $X$ and $S_E$ is the unit sphere of $E$. 
\smallskip

The \emph{Gelfand numbers} $c_n (S)$, $n \geq 1$, which are defined by:
\begin{equation} \label{Gelfand}
c_n (S) = \inf \{ \| S_{\mid M} \| \, ; \ {\rm codim}\, M < n\} \, .
\end{equation} 

The \emph{Kolmogorov numbers} $d_n (S)$, $n \geq 1$, which are defined by:
\begin{equation} \label{Kolmogorov}
d_n (S) = \inf _{\dim E < n} \raise -2pt \hbox{$\bigg[$} \sup_{x \in B_X} {\rm dist}\, (S x,  E) \raise -2pt \hbox{$\bigg]$} \, .
\end{equation} 

Pietsch showed that all $s$-numbers on Hilbert spaces are equal (see \cite{Pietsch}, \S~2, Corollary, or \cite{Pietsch-livre}, Theorem~11.3.4); hence:
\begin{equation} \label{egalite}
a_n (S) = b_n (S) = c_n (S) = d_n (S) \, .
\end{equation} 
%

We denote $m$  the normalized Lebesgue measure on $\T = \partial \D$. If $\varphi \colon \D \to \D$, $m_\varphi$ is the pull-back measure on 
$\overline{\D}$ defined by $m_\varphi (E) = m [{\varphi^\ast}^{- 1} (E)]$, where $\varphi^\ast$ stands for the non-tangential boundary values of $\varphi$. 
\smallskip

The notation $A \lesssim B$ means that $A \leq C \, B$ for some positive constant $C$ and we write $A \approx B$ if we have both $A \lesssim B$ and 
$B \lesssim A$.

\section{Preliminary results on weighted composition operators on $H^2 (\D)$} \label{sec: weighted}

We see in this section that the presence of a ``rapidly decaying'' weight allows simpler estimates for the approximation numbers of a corresponding weighted 
composition operator. Such a study, but a bit different, is made in \cite{Lechner-LQR}. 
\par\smallskip

Let $\varphi \colon \D\to \D$ a non-constant analytic self-map in the disk algebra $A(\D)$ such that, for some constant $C > 1$ and for all $z\in \D$:
\begin{equation}\label{veritas}
\varphi (1) = 1\, ,\quad |1 - \varphi (z)|\leq 1 \, , \quad |1 - \varphi(z)|\leq C \, (1 -|\varphi(z)|) 
\end{equation}
as well as $\varphi(z)\neq 1$ for $z\neq 1$. We can take for example $\varphi=\frac{1+ \lambda_{\theta}}{2}$ where $\lambda_\theta$ is the lens map 
with parameter $\theta$. 

Let $w \in H^\infty$ and let $T$ be the weighted composition operator 
\begin{displaymath} 
T = M_{w\circ \varphi} C_\varphi \colon H^2 \to H^2 \, . 
\end{displaymath} 
Note that $M_{w\circ \varphi} C_\varphi = C_\varphi M_w$. We first show that:
\begin{theorem} \label{fish} 
Let  $T = M_{w\circ \varphi} C_\varphi \colon H^2 \to H^2$ be as above and let $B$ be a Blaschke product with length $< N$. Then, with the implied 
constant depending only on the number $C$ in \eqref{veritas} (and of $\varphi$):
\begin{displaymath} 
a_{N}(T) \lesssim \sup_{|z - 1|\leq 1,\  z\in \varphi (\D)} |B (z)| \, |w(z)| \, .
\end{displaymath} 
\end{theorem} 
\begin{proof}
The following preliminary observation (see also \cite{LELIQR}, p.~809),  in which  we denote by $S (\xi, h)=\{z\in \D \, ; \ |z - \xi|\leq h\}$ the Carleson 
window with center $\xi\in \T$ and size $h$, and by $K_\varphi$  the support of the pull-back measure $m_{\varphi} $, will be useful.
\begin{equation}\label{prel} 
u\in S (\xi,h) \cap K_\varphi \quad \Longrightarrow  \quad u\in S (1, Ch) \cap K_\varphi \, .  
\end{equation}
Indeed, if $|u - \xi|\leq h$ and $u \in K_\varphi$, \eqref{veritas} implies:
\begin{displaymath} 
1 - |u|\leq  |u - \xi|\leq h \quad \text{and} \quad |u - 1|\leq C (1-|u|)\leq Ch \, .
\end{displaymath} 
Set $E = B H^2$.  This is a subspace of codimension $<N$. If $f = B g \in E$, with $\Vert g \Vert = \Vert f\Vert$ (isometric division by $B$ in $BH^2$), 
we have $Tf = (w B g) \circ \varphi$, whence:
\begin{displaymath} 
\Vert T(f) \Vert^2 = \int_{\D} |B|^2|w|^2 |g|^2 dm_{\varphi} \, ,
\end{displaymath} 
implying $\Vert T(f)\Vert^2\leq \Vert f\Vert^2 \Vert J\Vert^2$ where $J \colon H^2\to L^{2}(\sigma)$ is the natural embedding and where
\begin{displaymath} 
\sigma = |B|^2|w|^2  dm_{\varphi} \, . 
\end{displaymath} 
Now,  Carleson's embedding theorem for the measure $\sigma$ and \eqref{prel} show that (the implied constants being absolute):
\begin{align*}
\Vert J\Vert^2 
& \lesssim \sup_{\xi\in \T,\ 0<h<1} \frac{1}{h} \int_{S(\xi,h)\cap K_\varphi}  |B|^2|w|^2  dm_{\varphi} \\
& \lesssim \sup_{ 0<h<1} \frac{1}{h} \int_{S(1,Ch)\cap K_\varphi}  |B|^2|w|^2  dm_{\varphi} \\
& \lesssim\bigg(\sup_{|z-1|\leq 1,\ z\in \overline{\varphi(\D)}} 
|B(z)|^2 |w (z)|^2 \bigg)\, \bigg(\sup_{ 0<h<1} \frac{1}{h} \int_{S(1,Ch)\cap K_\varphi}  dm_{\varphi}\bigg) \\
&\lesssim  \sup_{|z - 1|\leq 1,\ z\in \overline{\varphi(\D)}} |B(z)|^2|w(z)|^2 \, ,
\end{align*}
since  $m_\varphi$ is a Carleson measure for $H^2$ and where we used that, according to \eqref{veritas}:
\begin{displaymath} 
K_\varphi \subseteq \overline{\varphi(\D)} \subseteq \{z\in \D \, ; \ |z - 1|\leq 1\} \, .
\end{displaymath} 

This ends the proof of Theorem~\ref{fish} with help of  the equality of $a_{N}(T)$ with the Gelfand number $c_N (T)$ recalled in \eqref{egalite}.
\end{proof}

In order to specialize efficiently the general Theorem~\ref{fish}, we recall the following simple Lemma~2.3 of \cite{LELIQR}, where:
\begin{equation} \label{distance pseudo}
\qquad \qquad \rho (a, b) = \bigg| \frac{a - b}{1 - \bar{a} b} \bigg| \, , \qquad a, b \in \D \, ,
\end{equation} 
is the \emph{pseudo-hyperbolic distance}:
\begin{lemma} [\cite{LELIQR}] \label{rec} 
Let $a, b\in \D$ such that $|a - b|\leq L \min (1 - |a|, 1 -|b|)$. Then:
\begin{displaymath} 
\rho (a, b) \leq \frac{L}{\sqrt{L^2 +1}} =: \kappa < 1 \, .
\end{displaymath} 
\end{lemma} 
We can now state:
\begin{theorem} \label{special} 
Assume that $\varphi$ is as in \eqref{veritas} and that the weight $w$ satisfies, for some parameters $0< \theta \leq 1$ and $R > 0$:
\begin{displaymath} 
\qquad |w(z)|\leq \exp\bigg( - \frac{R}{|1 - z|^\theta}\bigg) \, , \quad \forall z\in \D \text{ with } \Re z \geq 0 \, .
\end{displaymath} 
Then, the approximation numbers of $T = M_{w\circ \varphi} C_\varphi$ satisfy:
\begin{displaymath} 
a_{nm+1} (T) \lesssim \max \big[\exp(-a n), \exp(- R \, 2^{m\theta})\big] \, ,
\end{displaymath} 
for all integers $n, m\geq 1$, where $a = \log [\sqrt{16C^2 +1} / (4C)] > 0$ and $C$ is as in \eqref{veritas}.
\end{theorem} 
\begin{proof}
Let $p_l = 1 - 2^{- l}$, $0 \leq l < m$ and let $B$ be the Blaschke product:
\begin{displaymath} 
B (z) = \prod_{0 \leq l < m} \bigg(\frac{z - p_l}{1 - p_{l} z}\bigg)^n \, . 
\end{displaymath} 
Let $z \in K_\varphi \cap \D$ so that $0<|z - 1|\leq 1$. 
Let $l$ be the non-negative integer such that $2^{- l - 1}< |z - 1| \leq 2^{- l}$.  We separate two cases:
\smallskip

\noindent{\sl Case 1: $l\geq m$}. Then, \emph{the weight does the job}. Indeed, majorizing $|B (z)|$ by $1$  and using the assumption on $w$, we get:
\begin{align*}
|B(z)|^{2}|w(z)|^2 
&\leq |w(z)|^2 \leq \exp \Big(- \frac{2R}{|1 - z|^{\theta}} \Big) \\ 
& \leq \exp (- 2 R \, 2^{l\theta}) \leq \exp (- 2 R \, 2^{m\theta}) \, .
\end{align*}

\noindent {\sl Case 2: $l < m$}. Then, \emph{the Blaschke product does the job}. Indeed, majorize $|w(z)|$ by $1$ and estimate $|B(z))|$ more accurately 
with help of Lemma~\ref{rec}; we observe that 
\begin{displaymath} 
|z - p_l| \leq |z - 1|+ 1 - p_l \leq 2 \times 2^{- l} = 2 (1 - p_l) \leq 4 C (1-p_l) 
\end{displaymath} 
and then, since $z\in K_\varphi$, we can write  with   $C\geq 1$ as in \eqref{veritas}:
\begin{displaymath} 
1 -|z|\geq \frac{1}{C} \, |1 - z|\geq  \frac{1}{2 C} \, 2^{- l}\geq \frac{1}{4 C} \, |z - p_l| \, ,  
\end{displaymath} 
so that the assumptions of Lemma~\ref{rec} are verified with $L = 4C$, giving:
\begin{displaymath} 
\rho (z, p_l) \leq  \frac{4 C}{\sqrt{16 C^2 +1}} = \exp (- a) < 1 \, . 
\end{displaymath} 
Hence, by definition, since $l< m$: 
\begin{displaymath} 
|B (z)| \leq [ \rho (z, p_l) ]^n \leq \exp(- a n) \, .
\end{displaymath} 

Putting both cases together, and observing that our Blaschke product has length $n m < n m+1$, we get the result by applying Theorem~\ref{fish} with 
$N = n m+1$. 
\end{proof}
%

\subsection{Some remarks}

{\bf 1.} Twisting a composition operator by a weight may improve the compactness of this composition operator, or even may make this weighted composition 
operator compact though the non-weighted was not (see \cite{GKP} or \cite{Lechner-LQR}). However, this is not possible for all symbols, as seen in the 
following proposition.
\begin{proposition} \label{evid}
Let $w \in H^\infty$. If $\varphi$ is inner, or more generally if $|\varphi| = 1$ on a subset of $\T$ of positive measure, then $M_w \, C_\varphi$ is never 
compact (unless $w \equiv 0$).
\end{proposition}
\begin{proof}
Indeed, suppose $T = M_w \, C_\varphi$ compact. Since $(z^n)_n$ converges weakly to $0$ in $H^2$ and since $T (z^n) = w \, \varphi^{n}$, we should have, 
since $|\varphi|= 1$ on $E$, with $m (E) > 0$: 
\begin{displaymath} 
\int_{E} |w|^2 \, dm = \int_{E} |w|^2 |\varphi|^{2 n} \, dm 
\leq \int_{\T} |w|^2 |\varphi|^{2 n} \, dm =  \Vert T (z^n) \Vert^2  \converge_{n \to \infty} 0 \, ,
\end{displaymath} 
but this would imply that $w$ is null a.e. on $E$ and hence $w \equiv 0$ (see \cite{Duren}, Theorem~2.2), which was excluded.
\end{proof} 

Note that \'E.~Amar and A.~Lederer proved in \cite{Amar} that $|\varphi| = 1$ on a set of positive measure if and only if 
$\varphi$ is an exposed point of  of the unit ball of $H^\infty$; hence the following proposition can be viewed as the (almost) opposite case.

\begin{proposition} \label{exposed}
Let $\varphi \colon \D \to \D$ such that $\| \varphi \|_\infty = 1$. Assume that: 
\begin{displaymath} 
\int_\T \log (1 - |\varphi|) \, dm > - \infty 
\end{displaymath} 
(meaning that $\varphi$ is not an extreme point of the unit ball of $H^\infty$: see \cite{Duren}, Theorem~7.9). Then, if $w$ is an outer function such that 
$|w| = 1 - |\varphi|$, the weighted composition operator $T = M_w C_\varphi$ is Hilbert-Schmidt.
\end{proposition}
\begin{proof}
We have:
\begin{displaymath} 
\sum_{n = 0}^\infty \| T (z^n) \|^2 = \sum_{n = 0}^\infty \int_\T (1 - |\varphi|)^2 |\varphi|^{2 n} \, dm 
= \int_\T \frac{1 - |\varphi|}{1 + |\varphi|} \, dm < + \infty \, ,
\end{displaymath} 
and $T$ is Hilbert-Schmidt, as claimed.
\end{proof}

\medskip
{\bf 2.} In \cite{Lechner-LQR}, Theorem~2.5, it is proved that we always have, for some constants $\delta, \rho > 0$:
\begin{equation} \label{mino generale}
\qquad a_{n}( M_w C_\varphi) \geq \delta \, \rho^n \, , \quad n = 1, 2,  \dots
\end{equation} 
(if $w \not\equiv 0$). We give here an alternative proof, based on a result of Gunatillake (\cite{GUN}), this result holding in a wider context.  
\begin{theorem} [Gunatillake] \label{gaj} 
Let $T = M_w C_\varphi$ be a compact weighted composition operator on $H^2$ and assume that $\varphi$ has a fixed point $a\in \D$. Then the spectrum of 
$T$ is the set:
\begin{displaymath} 
\sigma (T) =\{0, w(a), w(a) \, \varphi'(a), w(a) \, [\varphi '(a)]^2, \ldots, w(a) \, [\varphi '(a)]^n, \ldots \}
\end{displaymath} 
\end{theorem} 
\begin{proof} [Proof of \eqref{mino generale}] 
First observe that, in view of Proposition~\ref{evid}, $\varphi$ cannot be an automorphism of $\D$ so that the point $a$ is the Denjoy-Wolff point 
of $\varphi$ and is attractive. Theorem~\ref{gaj} is interesting only when $w(a) \, \varphi ' (a)\neq 0$. 

Now, we can give a new proof Theorem 2.5 of \cite{Lechner-LQR}  as follows. Let $a\in \D$ be such that $w(a) \, \varphi ' (a) \neq 0$ ($H(\D)$ is a division 
ring and $\varphi ' \not\equiv 0$, $w \not\equiv 0$). Let $b = \varphi (a)$ and $\tau \in {\rm Aut}\, \D$ with $\tau (b) = a$. We set:
\begin{displaymath} 
\psi = \tau \circ \varphi \quad  \text{and} \quad  S = M_w C_\psi = T C_\tau \, . 
\end{displaymath} 
This operator $S$ is compact because $T$ is. 

Since $\psi (a) = a$ and $\psi ' (a) = \tau '(b) \varphi ' (a) \neq 0$,  Theorem~\ref{gaj} says that the non-zero 
eigenvalues of $S$, arranged in non-increasing order, are the numbers $\lambda_n = w (a) \, [\psi ' (a)]^{n - 1}$,  $n\geq 1$. As a consequence of 
Weyl's inequalities, we know that:
\begin{displaymath} 
a_{1} (S) \, a_{n}(S) \geq |\lambda_{2 n}|^2 \geq \delta \, \rho^{n} \, ,
\end{displaymath} 
with:
\begin{displaymath} 
\delta = |w (a)|^{2} > 0 \quad \text{and} \quad \rho =|\psi ' (a)|^{4} > 0 \, .
\end{displaymath} 
To finish, it is enough to observe that $a_{n} (S) \leq a_{n} (T) \, \Vert C_\tau \Vert$ by the ideal property of approximation numbers.
\end{proof}
%
 
\section{The splitted case} \label{sec: splitted}

\begin{theorem}\label{treprov} 
Let $\Phi = (\phi, \psi) \colon \D^d \to \D^d$ be a truly $d$-dimensional symbol with $\phi \colon \D\to \D$ depending only on $z_1$ and 
$\psi \colon\D^{d-1}\to \D^{d - 1}$ only on $z_2,\ldots, z_d$, i.e. $\Phi (z_1, z_2, \ldots, z_d) = \big( \phi (z_1), \psi (z_2, \ldots, z_d)\big)$. Then, 
whatever $\psi$ behaves:
\begin{displaymath} 
\Vert \phi \Vert_\infty = 1 \quad \Longrightarrow \quad \beta_{d}(C_\Phi) =1 \, .
\end{displaymath} 
\end{theorem}

\begin{proof} 
The proof is based on the following simple lemma, certainly well-known. 
\begin{lemma}\label{alegro} 
Let $S \colon H_1\to H_1$ and $T\colon H_2\to H_2$ be two compact linear operators, where $H_1$ and $H_2$ are Hilbert spaces. Let $S\otimes T$ be 
their tensor product, acting on the tensor product $H_1\otimes H_2$. Then:
\begin{displaymath} 
a_{m n}(S\otimes T)\geq a_{m}(S) \, a_{n}(T)
\end{displaymath} 
for all positive integers $m, n$.
\end{lemma}

We postpone the proof of the lemma and show how to conclude. \smallskip

We can assume $C_\Phi$ to be compact, so that $C_{\phi}$ is compact as well.  Since $\Vert \phi \Vert_\infty = 1$, we have, thanks to 
\eqref{equiv-dim1}\,:
\begin{displaymath} 
\qquad a_{m} (C_{\phi}) \geq \e^{- m \, \eps_m} \quad \text{with} \quad \eps_m \converge_{m \to \infty} 0 \, .
\end{displaymath} 
Replacing $\eps_m$ by $\delta_m:= \sup_{p\geq m} \eps_p$, we can assume that $(\eps_m)_m$ is non-increasing. Moreover,  
\begin{displaymath} 
m \, \eps_m\to \infty 
\end{displaymath} 
since $C_{\phi}$ is compact and hence $a_{m} (C_{\phi}) \converge_{m \to \infty} 0$.  We next observe that, due to the separation of variables in the 
definition of $\phi$ and $\psi$, we can write:
\begin{equation}\label{proten} 
C_\Phi = C_{\phi}\otimes C_\psi \, .
\end{equation}
Indeed, write $z = (z_1, w)$ with $z_1 \in \D$ and $w \in \D^{d-1}$. If $f \in H^{2}(\D)$ and $g\in H^{2}(\D^{d-1})$, we see that:
\begin{align*}
C_{\Phi} (f\otimes g) (z) 
& = (f\otimes g) \big( \phi (z_1), \psi (w) \big) = f \big( \phi (z_1) \big)\, g \big( \psi(w) \big) \\
& = [C_{\phi} f (z_1)] \,  [C_{\psi} g (w)] = (C_{\phi} f \otimes C_{\psi}g) (z) \, .
\end{align*}
Since tensor products $f \otimes g$ generate $H^{2}(\D^d)=H^2 (\D) \otimes H^2 (D^{d - 1})$, this proves \eqref{proten}. 
\par\smallskip

Let now $m$ be a large positive integer. Set ($[\, . \, ]$ stands for the integer part):
\begin{equation} 
n_m = [m\eps_m]^{d-1} \quad \text{and} \quad N_m =m \, n_m \, .
\end{equation} 
From what we know in dimension  $d - 1$ (see \cite{BLQR}, Theorem~3.1) and from the preceding, we can write (observe that $\psi$ has to be truly 
$(d-1)$-dimensional since $\Phi$ is truly $d$-dimensional):
\begin{displaymath} 
a_{m} (C_{\phi})\geq \exp(- m\, \eps_m) \quad \text{and}\quad  a_{n} (C_{\psi})\geq a \, \exp (- C \, n^{1/(d-1)}) \, ,
\end{displaymath} 
for some positive constant $C$, which will be allowed to vary from one formula to another. Lemma~\ref{alegro} implies: 
\begin{displaymath} 
a_{N_m} (C_\Phi)\geq a \, \exp [- C \, (m \, \eps_m + {n_m}^{1/(d-1)}) ] \, .
\end{displaymath} 
Since $n_m \lesssim (m \eps_m)^{d - 1}$, we get:
\begin{displaymath} 
a_{N_m} (C_\Phi)\geq a \, \exp(- C \, m \, \eps_m) \, .
\end{displaymath} 
Observe that $N_m = m \, n_m \sim m^d {\eps_m}^{d-1}$ and so ${N_m}^{1/d} \sim m \, {\eps_m}^{1 - 1/ d}$. 
As a consequence:
\begin{align*} 
a_{N_m} (C_\Phi) \geq a \, \exp (- C \, m\,\eps_m ) 
& = a \, \exp \big[- (C \, {\eps_m}^{1 / d})\, (m \, {\eps_m}^{1- 1 /d})\big]  \\
& \geq a \, \exp (-\eta_{m}\, {N_m}^{1/d} )
\end{align*} 
with $\eta_m: =C \,{\eps_m}^{1/d}$. \smallskip

Now, for $N > N_1$, let $m$ be the smallest integer satisfying  $N_m \geq N$ (so that $N_{m - 1}< N \leq N_{m}$), and set  $\delta_N =\eta_m$.  
We have  $\lim_{N\to \infty} \delta_N =0$. Next, we note that $\lim_{m\to \infty} N_{m} / N_{m -1} = 1$, because $N_m \geq N_{m - 1}$ and:
\begin{displaymath} 
\frac{N_{m}}{N_{m-1}} \leq \frac{m}{m - 1} \, \bigg( \frac{m \, \eps_m + 1}{(m - 1)\, \eps_{m - 1}}\bigg)^{d - 1} 
\sim \bigg(\frac{\eps_{m}}{\eps_{m-1}}\bigg)^{d-1} \leq 1 \, . 
\end{displaymath} 

Finally, if $N$ is an arbitrary integer and $N_{m-1} < N \leq N_m$, we obtain:
\begin{displaymath} 
a_{N} (C_\Phi)\geq a_{N_m} (C_\Phi) \geq a \, \exp (-\eta_m \, {N_m}^{1/d}) \geq a \, \exp (- C \, \delta_N N^{1/d}) \, ,
\end{displaymath} 
since we observed that $\lim_{m\to \infty} N_{m} / N_{m-1}=1$. 
\smallskip

This amounts to say that  $\beta_{d}(C_\Phi)=1$.
\end{proof}
\begin{proof}[Proof of Lemma~\ref{alegro}.]
It is rather formal. Start from the Schmidt decompositions of $S$ and $T$ respectively (recall that Hilbert spaces, the approximation numbers are equal to the 
singular ones):
\begin{displaymath} 
S =\sum_{m=1}^\infty a_{m}(S) \, u_m \odot v_m \, ,\qquad T=\sum_{n=1}^\infty a_{n}(T) \, u'_n \odot v'_n \, ,
\end{displaymath} 
where $(u_m)$, $(v_m)$ are two orthonormal sequences of $H_1$,  $(u'_n)$, $(v'_n)$ two orthonormal sequences of $H_2$, and 
$u_m\odot v_m$ and $u'_n \odot v'_n$ denote the rank one operators defined by $(u_m \odot v_m) (x)=\langle x, v_m\rangle \, u_m$, $x \in H_1$,  
and $(u'_n \odot v'_n) (x) = \langle x, v'_n \rangle \, u'_n$, $x \in H_2$. 

We clearly have:
\begin{displaymath} 
(u_m \odot v_m) \otimes (u'_n \odot v'_n) = (u_m\otimes u'_n) \odot (v_m\otimes v'_n) \, ,
\end{displaymath} 
so that the Schmidt decomposition of $S\otimes T$ is (with SOT-convergence):
\begin{displaymath} 
S\otimes T = \sum_{m, n \geq 1} a_{m}(S) \, a_{n}(T) \, (u_m \otimes u'_n) \odot (v_m \otimes v'_n) \,,
\end{displaymath} 
since the two sequences $(u_m\otimes u'_n)_{m, n}$ and $(v_m\otimes v'_n)_{m, n}$ are orthonormal: for instance, we have by definition:
\begin{displaymath} 
\langle u_{m_1} \otimes u'_{n_1}, u_{m_2}\otimes u'_{n_2}\rangle = \langle u_{m_1}, u_{m_2},\rangle\, \langle u'_{n_1}, u'_{n_2}\rangle.
\end{displaymath} 
This shows that the singular values of $S\otimes T$ are the non-increasing rearrangement of the positive numbers $a_{m}(S) \, a_{n}(T)$ and ends the proof of 
the lemma: the $mn$ numbers $a_{k}(S)\, a_{l}(T)$, for $1\leq k\leq m$, $1\leq l\leq n$ all satisfy $a_{k}(S)\, a_{l}(T)\geq  a_{m}(S) \, a_{n}(T)$, 
so that $a_{mn}(S\otimes T)\geq a_{m}(S) \, a_{n}(T)$.
\end{proof}

\section{The glued case} \label{sec: glued}

Here  we consider symbols of the form:
\begin{equation} \label{glued map}
\Phi (z_1, z_2) = \big( \phi (z_1), \phi (z_1)\big) \, ,
\end{equation} 
where $\phi \colon \D \to \D$ is a non-constant analytic map. 
\smallskip

Note that such maps $\Phi$ are not truly $2$-dimensional.

\subsection{Preliminary} 

We begin by remarking the following fact.
\smallskip

Let $B^2 (\D)$ be the Bergman space of all analytic functions $f \colon \D \to \C$ such that:
\begin{displaymath} 
\| f \|_{B^2}^2 := \int_\D |f (z)|^2 \, dA (z) < \infty \, ,
\end{displaymath} 
where $dA$ is the normalized area measure on $\D$. 

\begin{proposition} \label{B2-H2}
Assume that the composition operator $C_\phi$ maps boundedly $B^2 (\D)$ into $H^2 (\D)$. Then $C_\Phi \colon H^2 (\D^2) \to H^2 (\D^2)$, 
defined by \eqref{glued map}, is bounded.
\end{proposition}
\begin{proof}
If we write $f \in H^2 (\D^2)$ as:
\begin{displaymath} 
f (z_1, z_2) = \sum_{j, k \geq 0} c_{j, k} z_1^j z_2^k \, , \quad \text{with} \quad \sum_{j, k \geq 0} |c_{j, k}|^2 = \| f \|_{H^2}^2 \, ,
\end{displaymath} 
we formally (or assuming that $f$ is a polynomial) have:
\begin{displaymath} 
[C_{\Phi} f ] (z_1, z_2) = \sum_{j, k \geq 0} c_{j, k}[\phi (z_1)]^j [\phi (z_1)]^k 
= \sum_{n = 0}^\infty \bigg(\sum_{j + k =n} c_{j, k} \bigg) [\phi (z_1)]^n \, .
\end{displaymath} 
Hence, if we set $g (z) = \sum_{n = 0}^\infty \big( \sum_{j + k = n} c_{j, k} \big) z^n$, we get:
\begin{displaymath} 
[C_{\Phi} (f ) ] (z_1, z_2) = [C_\phi (g)] (z_1) \, , 
\end{displaymath} 
so that, by integrating:
\begin{displaymath} 
\| C_\Phi (f) \|_{H^2 (\D^2)} = \| C_\phi (g) \|_{H^2 (\D)} \, .
\end{displaymath} 
By hypothesis, there is a positive constant $M$ such that:
\begin{displaymath} 
\| C_\phi (g) \|_{H^2 (\D)} \leq M \, \| g \|_{B^2 (\D)} \, . 
\end{displaymath} 
But, by the Cauchy-Schwarz inequality: 
\begin{align*} 
\| g \|_{B^2 (\D)}^2 
& = \sum_{n = 0}^\infty \frac{1}{n + 1}\, \bigg| \sum_{j + k= n} c_{j, k} \bigg|^2 \\
& \leq \sum_{n = 0}^\infty \bigg( \sum_{j + k = n} |c_{j, k}|^2 \bigg) = \sum_{j, k \geq 0} |c_{j, k}|^2 = \| f \|_{H^2 (\D^2)}^2 \, ,
\end{align*} 
and we obtain $\| C_\Phi (f) \|_{H^2 (\D^2)} \leq M \, \| f \|_{H^2 (\D^2)}$. 
\end{proof}

\subsection{Lens maps}

Let $\lambda_\theta$ be a lens map of parameter $\theta$, $0 < \theta < 1$. 
We consider $\Phi_\theta \colon \D^2 \to \D^2$ defined by:
\begin{equation} 
\Phi_\theta (z_1, z_2) = \big( \lambda_\theta (z_1), \lambda_\theta (z_1) \big) \, .
\end{equation} 

We have the following result.
\begin{theorem} \label{bi-lens}
The composition operator $C_{\Phi_\theta} \colon H^2 (\D^2) \to H^2 (\D^2)$ is: 
\par\smallskip

$1)$ not bounded for $\theta > 1/2$; 
\par\smallskip

$2)$ bounded, but not compact for $\theta = 1/2$;
\par\smallskip

$3)$ compact, and even Hilbert-Schmidt, for $0 < \theta < 1/2$.
\end{theorem} 
\begin{proof}
The reproducing kernel of $H^2 (\D^2)$ is, for $(a, b) \in \D^2$:
\begin{equation} 
\qquad \qquad \qquad K_{a, b} (z_1, z_2) = \frac{1}{1 - \bar{a} z_1} \, \frac{1}{1 - \bar{b} z_2} \, \raise 1,5 pt \hbox{,} \qquad (z_1, z_2) \in \D^2 \, ,
\quad\quad
\end{equation} 
and:
\begin{displaymath} 
\| K_{a, b} \|^2 = \frac{1}{(1 - |a|^2) (1 - |b|^2)} \, \cdot 
\end{displaymath} 

$1)$ If $C_{\Phi_\theta}$ were bounded, we should have, for some $M < \infty$:
\begin{displaymath} 
\qquad \qquad \| C_{\Phi_\theta}^{\, \ast} (K_{a, b} ) \|_{H^2} \leq M \, \| K_{a, b} \|_{H^2} \, , \quad \text{for all } a, b \in \D \, .
\end{displaymath} 
Since $C_{\Phi_\theta}^{\, \ast} (K_{a, b} ) = K_{\Phi_\theta (a, b)} = K_{\lambda_\theta (a), \lambda_\theta (a)}$, we would get, with $b = 0$:
\begin{displaymath} 
\bigg(\frac{1}{1 - |\lambda_\theta (a)|^2}\bigg)^2 \leq M^2 \, \frac{1}{1 - |a|^2} \, ;
\end{displaymath} 
but this is not possible for $\theta > 1/2$, since $1 - |\lambda_\theta (a)|^2 \approx 1 - |\lambda_\theta (a)| \sim (1 - a)^\theta$ when $a$ goes to $1$, 
with  $0 < a < 1$. 
\par\smallskip

For $2)$ and $3)$, let us consider the pull-back measure $m_\theta$ of the normalized Lebesgue measure on $\T = \partial \D$ by $\lambda_\theta$. It is easy  
to see that:
\begin{equation} 
\sup_{\xi \in \T} m_\theta [ D (\xi, h) \cap \D)] = m_\theta [D (1, h) \cap \D] \approx h^{1 /\theta} \, .
\end{equation} 
In particular, for $\theta \leq 1/2$, $m_\theta$ is a $2$-Carleson measure, and hence (see \cite{LELIQR-TAMS}, Theorem~2.1, for example) the canonical 
injection $j \colon B^2 (\D) \to L^2 (m_\theta)$ is bounded, meaning that, for some positive constant $M < \infty$:
\begin{displaymath} 
\int_\D |f (z)|^2 \, d m_\theta (z) \leq M^2 \| f \|_{B^2}^2 \, .
\end{displaymath} 
Since
\begin{displaymath} 
\int_\D |f (z)|^2 \, dm_\theta (z) = \int_\T | f [\lambda_\theta (u)]|^2 \, dm (u) = \| C_{\lambda_\theta} (f) \|_{H^2}^2 \, ,
\end{displaymath} 
we get that $C_{\lambda_\theta}$ maps boundedly $B^2 (\D)$ into $H^2 (\D)$.\par
\smallskip

It follows from Proposition~\ref{B2-H2} that $C_{\Phi_\theta} \colon H^2 (\D^2) \to H^2 (\D^2)$ is bounded. 
\par\smallskip

However, $C_{\Phi_{1/2}}$ is not compact since $C_{\Phi_{1/2}}^{\, \ast} (K_{a, b}) / \| K_{a, b} \|$ does not converge to $0$ as 
$a , b \to 1$, by the calculations made in $1)$.\par
\smallskip

For $3)$, let $e_{j, k} (z_1, z_2) = z_1^j z_2^k$, $j, k \geq 0$, be the canonical orthonormal basis of $H^2 (\D^2)$; we have 
$[C_{\phi_\theta} (e_{j, k})] (z_1, z_2)  = [\lambda_\theta (z_1)]^{j + k}$. Hence:
\begin{displaymath} 
\sum_{j, k \geq 0} \| C_{\phi_\theta} (e_{j, k}) \|_{H^2 (\D^2)}^2 \leq \sum_{n = 0}^\infty (2 n + 1) \int_\T |\lambda_\theta|^{2 n} \, dm 
\leq \int_\T \frac{2}{(1 - |\lambda_\theta|^2 )^2} \, dm \, .
\end{displaymath} 
Since, by Lemma~\ref{estim-lens} below, $1 - |\lambda_\theta (\e^{it })|^2 \gtrsim |1 - \e^{i t}|^\theta \geq t^\theta$ for $| t | \leq \pi/2$, we get:
\begin{displaymath} 
\sum_{j, k \geq 0} \| C_{\phi_\theta} (e_{j, k}) \|_{H^2 (\D^2)}^2 \lesssim \int_0^{\pi / 2} \frac{dt}{t^{2 \theta}} < \infty,
\end{displaymath} 
since $\theta < 1/2$. Therefore $C_{\phi_\theta}$ is Hilbert-Schmidt for $\theta < 1/2$. 
\end{proof}

For sake of completeness, we recall the following elementary fact (see \cite{Shapiro-livre}, p.~28, or also \cite{LELIQR}, Lemma~2.5)).
\begin{lemma} \label{estim-lens}
With $\delta = \cos (\theta \pi / 2)$, we have, for $| z | \leq 1$ and $\Re z \geq 0$:
\begin{displaymath} 
1 - |\lambda_\theta (z)|^2 \geq \frac{\delta}{2}\, |1 - z|^\theta \, .
\end{displaymath} 
\end{lemma}
\begin{proof}
We can write:
\begin{displaymath} 
\lambda_\theta (z)=\frac{1 - w}{1 + w} \quad \text{with} \quad  w=\bigg(\frac{1 - z}{1 + z}\bigg)^\theta \quad \text{and } |w|\leq 1 \, . 
\end{displaymath} 
Then:
\begin{displaymath} 
\Re w \geq \delta \, |w|\geq \frac{\delta}{2}|1 - z|^\theta \, .
\end{displaymath} 
Hence:
\begin{displaymath} 
1- |\lambda_\theta (z)|^2 = \frac{4 \, \Re w}{|1 + w|^2} \geq \delta \, |w|\geq \frac{\delta}{2} \, |1 - z|^\theta \, ,
\end{displaymath} 
as announced
\end{proof}
\medskip

We now improve the result $3)$ of Theorem~\ref{bi-lens} by estimating the approximation numbers of $C_{\Phi_\theta}$ and get that 
$C_{\Phi_\theta}$ is in all Schatten classes of $H^2 (\D^2)$ when $\theta < 1/2$.
\begin{theorem} 
For $0 < \theta < 1/2$, there exists $b = b_\theta > 0$ such that:
\begin{equation} 
a_n (C_{\Phi_\theta}) \lesssim \e^{- b \sqrt{n}} \, .
\end{equation} 
In particular $\beta_2^+ (C_{\Phi_\theta}) \leq \e^{- b} < 1$, though $\| \Phi_\theta \|_\infty = 1$, and even 
$\Phi_\theta (\T^2) \cap \T^2 \neq \emptyset$.
\end{theorem} 
\begin{proof}
Proposition~\ref{B2-H2} (and its proof) can be rephrased in the following way: if $C_\phi$ maps boundedly $B^2 (\D)$ into $H^2 (\D)$, then, we have the 
following factorization:
\begin{equation} 
C_\Phi \colon H^2 (\D^2) \mathop{\longrightarrow}^J B^2 (\D) \mathop{\hbox to 20 pt {\rightarrowfill}}^{C_\phi} H^2 (\D) 
\mathop{\longrightarrow}^I H^2 (\D^2) \, ,
\end{equation} 
where $I \colon H^2 (\D) \to H^2 (\D^2)$ is the canonical injection given by $(If) (z_1, z_2) = f (z_1)$ for $f \in H^2 (\D)$, and 
$J \colon H^2 (\D^2) \to B^2 (\D)$ is the contractive map defined by: 
\begin{displaymath} 
(J f) (z) = \sum_{n = 0}^\infty \bigg( \sum_{j + k = n} c_{j, k} \bigg) z^n \, ,
\end{displaymath} 
for $f \in H^2 (\D^2)$ with $f (z_1, z_2) = \sum_{j, k \geq 0} c_{j, k} z_1^j z_2^k$.
\par\smallskip

In the proof of Theorem~\ref{bi-lens}, we have seen that, for $0 < \theta \leq 1/2$, the composition operator $C_{\lambda_\theta}$ is bounded from 
$B^2 (\D)$ into $H^2 (\D)$; we get hence the factorization:
\begin{displaymath} 
C_{\Phi_\theta} \colon H^2 (\D^2) \mathop{\longrightarrow}^J B^2 (\D) \mathop{\hbox to 25 pt {\rightarrowfill}}^{C_{\lambda_\theta}} H^2 (\D) 
\mathop{\longrightarrow}^I H^2 (\D^2) \, ,
\end{displaymath} 

Now, the lens maps have a semi-group property:
\begin{equation} 
\lambda_{\theta_1 \theta_2} = \lambda_{\theta_1} \lambda_{\theta_2} \, ,
\end{equation} 
giving $C_{\lambda_{\theta_1 \theta_2}} = C_{\lambda_{\theta_1}} \circ C_{\lambda_{\theta_2}}$.
\par\smallskip

For $0 < \theta < 1/2$, we therefore can write $C_{\lambda_\theta} = C_{\lambda_{2 \theta}} \circ C_{\lambda_{1/2}}$ (note that $2 \theta < 1$, so 
$C_{\lambda_{2 \theta}} \colon H^2 (\D) \to H^2 (\D)$ is bounded), and we get:
\begin{displaymath} 
C_{\Phi_\theta} = I \, C_{\lambda_{2 \theta}} C_{\lambda_{1/2}} \, J \, .
\end{displaymath} 
Consequently:
\begin{displaymath} 
a_n (C_{\Phi_\theta}) \leq \| I \| \, \|J \| \, \| C_{\lambda_{1/2}} \|_{B^2 \to H^2} \, a_n (C_{\lambda_{2 \theta}}) \, .
\end{displaymath} 

Now, we know (\cite{LELIQR}, Theorem~2.1) that $a_n (C_{\lambda_{2 \theta}}) \lesssim \e^{- b \sqrt{n}}$, so we get that 
$a_n (C_{\Phi_\theta}) \lesssim \e^{- b \sqrt{n}}$. 
\end{proof}
\smallskip

\noindent{\bf Remark.} In \cite{BLQR}, we saw that for a truly $2$-dimensional symbol $\Phi$, we have $\beta_2^- (C_\phi) > 0$. Here the symbol 
$\Phi_\theta$ is not truly $2$-dimensional, but we nevertheless have $\beta_2 (C_{\Phi_\theta}) > 0$. In fact, let 
$E = \{ f \in H^2 (\D^2) \, ; \ \frac{\partial f}{\partial z_2} \equiv 0\}$; $E$ is isometrically isomorphic to $H^2 (\D)$ and the restriction of 
$C_{\Phi_\theta}$ to $E$ behaves as the $1$-dimensional composition operator $C_{\lambda_\theta} \colon H^2 (\D) \to H^2 (\D)$; hence (\cite{LIQUEROD}, Proposition~6.3):
\begin{displaymath} 
\e^{- b_0 \sqrt{n}} \lesssim a_n (C_{\lambda_\theta}) = a_n ({C_{\Phi_\theta}}_{\mid E}) \leq a_n (C_{\Phi_\theta}) \, ,
\end{displaymath} 
and $\beta_2^- (C_{\Phi_\theta}) \geq \e^{- b_0} > 0$.

\section{Triangularly separated variables} \label{sec: triangularly}

In this section, we consider symbols of the form:
\begin{equation} \label{triangular-symbol}
\Phi (z_1, z_2) = \big( \phi (z_1), \psi (z_1) \, z_2\big) \, ,
\end{equation} 
where $\phi , \psi\colon \D \to \D$ are non-constant analytic maps. 
\par\smallskip

Such maps $\Phi$ are truly $2$-dimensional. 
\par\smallskip

More generally, if $h \in H^\infty$, with $h (0) = 0$ and $\| h \|_\infty \leq 1$, has its powers $h^k$, $k \geq 0$, orthogonal in $H^2$ (for convenience, 
we shall say that $h$ is a \emph{Rudin function}), we can consider:
\begin{equation} \label{triangular-symbol-with-h}
\Phi (z_1, z_2) = \big( \phi (z_1), \psi (z_1) \, h (z_2) \big) \,
\end{equation} 
For such $h$ we can take for example an inner function vanishing at the origin, but there are other such functions, as shown by C.~Bishop:
\medskip

\noindent{\bf Theorem} (Bishop \cite{Bishop}).  
{\it The function $h$ is a Rudin function if and only if the pull-back measure $\mu = \mu_h$ is radial and Jensen, i.e for every Borel set $E$:
\begin{displaymath} 
\qquad \quad \mu (\e^{i\theta} E)=\mu (E) \quad \text{and} \quad \int_{\overline{\D}} \log (1/|z|) \, d\mu (z) < \infty \, .
\end{displaymath} 
Conversely, for every probability measure $\mu$ supported by $\overline{\D}$, which is radial and Jensen, there exists $h$ in the unit ball of $H^\infty$, with 
$h (0) = 0$, such that $\mu = \mu_h$.}
\medskip

If we take for $\mu$ the Lebesgue measure of $\T$, we get an inner function. But, as remarked in \cite{Bishop}, we can take for $\mu$ the Lebesgue measure 
on the union $\T \cup (1/2)\T$, normalized in order that $\mu (T) = \mu \big( (1/2) \T \big) = 1/2$. Then the corresponding $h$ is not inner since 
$|h| = 1/2$ on a subset of $\T$ of positive measure. He also showed that $h (z) / z$ may be a non-constant outer function. Also, P.~Bourdon (\cite{Bourdon}) 
showed that the powers of $h$ are orthogonal if and only if its Nevanlinna counting function is almost everywhere constant on each circle centered on the origin.

\subsection{General facts} \label{general facts}

We first observe that if $f \in H^2 (\D^2)$ and:
\begin{displaymath} 
f (z_1, z_2) = \sum_{j, k \geq 0} c_{j, k} \,  z_1^j z_2^k \, , 
\end{displaymath} 
then we can write:
\begin{displaymath} 
f (z_1, z_2) = \bigg( \sum_{k \geq 0} f_k (z_1) \bigg) \, z_2^k
\end{displaymath} 
with:
\begin{displaymath} 
f_k (z_1) = \sum_{j \geq 0} c_{j, k} \,  z_1^j \, ,
\end{displaymath} 
and:
\begin{displaymath} 
\| f \|_{H^2 (\D^2)}^2 = \sum_{j, k \geq 0} |c_{j, k}|^2 = \sum_{k \geq 0} \| f_k \|_{H^2 (\D)}^2 \, .
\end{displaymath} 
That means that we have an isometric isomorphism:
\begin{displaymath} 
J \colon H^2 (\D^2) \longrightarrow \bigoplus_{k= 0}^\infty H^2 (\D) \, ,
\end{displaymath} 
defined by $J f = (f_k)_{k \geq 0}$.\par
\smallskip

Now, for symbols $\Phi$ as in \eqref{triangular-symbol}, we have:
\begin{displaymath} 
(C_\Phi f) (z_1, z_2) = \sum_{j, k \geq 0} c_{j, k} \, [\phi (z_1)]^j [\psi (z_1)]^k z_2^k \, ,
\end{displaymath} 
so that $J \, C_\Phi \, J^{- 1}$ appears as the operator $\bigoplus_k M_{\psi^k} C_\phi$ on $\bigoplus_k H^2 (\D)$, where $M_{\psi^k}$ is the multiplication 
operator by $\psi^k$:
\begin{displaymath} 
[(M_{\psi^k} C_\phi ) f_k ] (z_1) = [\psi (z_1)]^k \, [(f_k \circ \phi) (z_1)] \, .
\end{displaymath} 

When $\Phi$ is as in \eqref{triangular-symbol-with-h}, we have:
\begin{displaymath} 
(C_\Phi f) (z_1, z_2) = \sum_{j, k \geq 0} c_{j, k} \, [\phi (z_1)]^j [\psi (z_1)]^k [h (z_2)]^k \, ,
\end{displaymath} 
with:
\begin{displaymath} 
\| C_\Phi f\|^2 \leq \sum_{k = 0}^\infty \| T_k f_k\|^2 
\end{displaymath} 
and:
\begin{displaymath} 
T_k = M_{\psi^k} C_\phi \, ;
\end{displaymath} 
hence $J \, C_\Phi \, J^{- 1}$ appears as pointwise dominated by the operator $T = \oplus_k T_k$ on $\bigoplus_k H^2 (\D)$. This implies a 
factorization $C_\Phi = A T$ with $\| A \| \leq 1$, so that $a_n (C_\Phi) \leq a_n (T)$ for all $n \geq 1$.
\medskip

We recall the following elementary fact.
\begin{lemma}
Let $(H_k)_{k \geq 0}$ be a sequence of Hilbert spaces and $T_k \colon H_k \to H_k$ be bounded operators. Let $H = \bigoplus_{k = 0}^\infty H_k$ and 
$T \colon H \to H$ defined by $T x = (T_k x_k)_k$. Then:
\par\smallskip

$1)$ $T$ is bounded on $H$ if and only if $\sup_k \| T_k\| < \infty$; 
\par\smallskip

$2)$ $T$ is compact on $H$ if and only if each $T_k$ is compact and $\| T_k\| \converge_{k \to \infty} 0$.
\end{lemma}

Going back to the symbols of the form \eqref{triangular-symbol}, we have $\| M_{\psi^k} \| \leq \| \psi^k\|_\infty \leq 1$, since $\| \psi\|_\infty \leq 1$; 
hence $\| M_{\psi^k} C_\phi \| \leq \| C_\phi \|$ and the operator $(M_{\psi^k} C_\phi)_k$ is bounded on $\bigoplus_k H^2 (\D)$. Therefore 
$C_\Phi$ is bounded on $H^2 (\D^2)$. 
\medskip

For approximation numbers, we have the following two facts.

\begin{lemma} \label{lemme utile}
Let $T_k \colon H_k \to H_k$ be bounded linear operators between Hilbert spaces $H_k$, $k \geq 0$. Let $H = \bigoplus_k H_k$ and 
$T = (T_k)_k \colon H \to H$, assumed to be compact. Then, for every $n_1, \ldots, n_K \geq 1$, and $0 \leq m_1 < \cdots < m_K$, $K \geq 1$, we have:
\begin{equation} 
a_N (T) \geq \inf_{1 \leq k \leq K} a_{n_k} (T_{m_k}) \, ,
\end{equation} 
where $N = n_1 + \cdots + n_K$.
\end{lemma}
\begin{proof}
We use the Bernstein numbers $b_n$ (see \eqref{Bernstein}), which are equal to the approximation numbers (see \eqref{egalite}).

For $k = 1, \ldots, K$, there is an $n_k$-dimensional subspace $E_k$ of $H_{m_k}$ such that:
\begin{displaymath} 
\qquad b_{n_k} (T_{m_k}) \leq \| T_{m_k} x \| \, , \quad \text{for all } x\in S_{E_k} \, . 
\end{displaymath} 
Then $E = \bigoplus_{k =1}^K E_k$ is an $N$-dimensional subspace of $H$ and for every $x = (x_1, x_2, \ldots) \in E$, we have: 
\begin{align*} 
\| T x \|^2 
& = \sum_{k \leq K} \| T_{m_k} x_{m_k}\|^2 \geq \sum_{k \leq K} [b_{n_k} (T_{m_k})]^2 \, \|x_{m_k}\|^2 \\
& \geq \inf_{k \leq K} [ b_{n_k} (T_{m_k}) ]^2 \sum_{k \leq K} \| x_{m_k}\|^2 
= \inf_{k \leq K} [b_{n_k} (T_{m_k}) ]^2 \| x\|^2 \, ;
\end{align*} 
hence $b_N (T) \geq \inf_{k \leq K} b_{n_k} (T_{m_k}) $, and we get the announced result.
\end{proof}
\begin{lemma} \label{majo} 
Let $T = \bigoplus_{k=0}^\infty T_k$ acting on a Hilbertian sum $H = \bigoplus_{k = 0}^\infty H_k$. Let $n_0, \ldots, n_K$ be positive integers and 
$N = n_0 + \cdots + n_K - K$.  Then, the approximation numbers of  $T$ satisfy:
\begin{equation}\label{maxsup} 
a_{N}(T) \leq \max \raise -1 pt \hbox{$\Big($} \max_{0\leq k\leq K} a_{n_k}(T_k), \sup_{k > K} \Vert T_k \Vert \raise -1 pt \hbox{$\Big)$} \, .
\end{equation}
\end{lemma}
\begin{proof}
Denote by $S$ the right-hand side of \eqref{maxsup}. Let $R_k$, $0\leq k\leq K$ be operators on $H_k$ of respective rank $< n_k$ such that 
$\Vert T_k - R_k\Vert = a_{n_k}(T_k)$  and let $R =\bigoplus_{k = 0}^{K} R_k$. Then $R$ is an operator of rank $\leq n_0 + \cdots + n_K - K - 1 <  N$. If 
$f = \sum_{k = 0}^\infty f_k \in H$, we see that:
\begin{align*}
\Vert T f - R f \Vert^{2} 
& =\sum_{k = 0}^K \Vert T_{k} f_k - R_{k}f_k \Vert^{2} + \sum_{k > K} \Vert T_{k} f_k \Vert^{2} \\
& \leq \sum_{k = 0}^{K} a_{n_k} (T_k)^{2} \Vert f_k \Vert^2 + \sum_{k > K} \Vert T_{k} f_k \Vert^{2}
\leq S^2 \sum_{k = 0}^\infty \Vert f_k \Vert^2 = S^2 \Vert f \Vert^2 \, ,
\end{align*}
hence the result.
\end{proof}

We give now two corollaries of Lemma~\ref{majo}.
\smallskip

\noindent{\bf Example 1.} We first use lens maps. We get:
\begin{theorem} \label{majo lens}
Let $\lambda_\theta$ the lens map of parameter $\theta$ and let $\psi \colon \D \to \D$ such that $\|\psi \|_\infty := c < 1$ and $h$ a Rudin function. 
We consider:
\begin{displaymath} 
\Phi (z_1, z_2) = \big( \lambda_\theta (z_1), \psi (z_1) \, h (z_2) \big) \, .
\end{displaymath} 
Then, for some positive constant $\beta$, we have, for all $N \geq 1$:
\begin{equation} 
a_N (C_\Phi) \lesssim \e^{- \beta \, N^{1/3}} \, .
\end{equation} 
\end{theorem}
\begin{proof}
Let $T_k = M_{\psi^k} C_{\lambda_\theta}$. We have $\| T_k \| \leq c^k$, so $\sup_{k > K} \| T_k\| \leq c^K$. On the other hand, we have 
$a_n (T_k) \leq c^k \, a_n (C_{\lambda_\theta}) \leq a_n (C_{\lambda_\theta}) \lesssim \e^{- \beta_\theta \sqrt{n}}$ (\cite{LELIQR}, Theorem~2.1). 
Taking $n_0 = n_1 = \cdots = n_K = K^2$ in Lemma~\ref{majo}, we get:
\begin{displaymath} 
\max_{0 \leq k \leq K} a_{n_k} (T_k) \lesssim \e^{ - \beta_\theta K} \, .
\end{displaymath} 
Since $n_0 + \cdots + n_K - K \approx K^3$, we obtain $a_{K^3} \lesssim \e^{ - \beta  K}$, which gives the claimed result, by taking 
$\beta = \max \big (\beta_\theta, \log (1/ c) \big)$. 
\end{proof}
\smallskip

\noindent{\bf Example 2.} We consider the cusp map $\chi$. We have:
\begin{theorem} \label{majo cusp}
Let $\chi$ be the cusp map, $h$ a Rudin function, and $\psi$ in the unit ball of $H^\infty$, with $\| \psi \|_\infty := c < 1$. Let:
\begin{displaymath} 
\Phi (z_1, z_2) = \big( \chi (z_1), \psi (z_1) \, h (z_2) \big) \, .
\end{displaymath} 
Then, for positive constant $\beta$, we have, for all $N \geq 1$:
\begin{displaymath} 
a_N (C_\Phi) \lesssim \e^{ - \beta \sqrt{N} / \sqrt{\log N}} \, .
\end{displaymath} 
\end{theorem} 
\begin{proof}
Let $T_k = M_{\psi^k} C_\chi$. As above, we have $\sup_{k > K} \| T_k\| \leq c^K$. For the cusp map, we have 
$a_n (C_\chi) \lesssim \e^{ - \alpha n / \log n}$ (\cite{LIQURO-Estimates}, Theorem~4.3); hence $a_n (T_k) \lesssim \e^{ - \alpha n / \log n}$. We 
take $n_0 = n_1 = \cdots = n_K = K \, [\log K]$ (where $[\log K]$ is the integer part of $\log K$). Since $n_0 + \cdots + n_K \approx K^2 [\log K]$,  we get, 
for another $\alpha > 0$:
\begin{displaymath} 
a_{K^2 [\log K]} (C_\Phi) \lesssim \e^{- \alpha K} \, ,
\end{displaymath} 
which reads: $a_N (C_\Phi) \lesssim \e^{ - \beta \sqrt{N / \log N}}$, as claimed. 
\end{proof}

\subsection{Lower bounds}

In this subsection, we give lower bounds for approximation numbers of composition operators on $H^2$ of the bidisk, attached to a symbol $\Phi$  of the 
previous form $\Phi (z_1, z_2) = \big( \phi (z_1), \psi (z_1) \, h (z_2) \big)$ where $h$ is a Rudin function. The sharpness of those estimates will be discussed 
in the next subsection. We first need some lemmas in dimension one.

\begin{lemma} \label{mino capa}
Let $u, v \colon \D \to \D$ be two non-constant analytic self-maps and $T = M_v C_u \colon H^2 (\D) \to H^2 (\D)$ be the associated weighted 
composition operator. For $0 < r < 1$, we set $A = u (r \, \overline{\D})$ and $\Gamma = \exp \big( - 1 / \capa (A) \big)$. Then, for 
$0 < \delta \leq \inf_{|z| = r} |v (z)|$, we have:
\begin{equation} \label{mino1}
a_n (T) \gtrsim \sqrt{1 - r}\, \delta \, \Gamma^n \, .
\end{equation} 
\end{lemma}
\smallskip

In this lemma, $\capa (A)$ denotes the Green capacity of the compact subset $A \subseteq \D$ (see \cite{LQR}, \S~2.3 for the definition). 
\smallskip

For the proof, we need the following result (\cite{WID}, Theorem~7, p.~353).
\begin{theorem} [Widom] \label{widom}
Let $A$ be a compact subset of  $\mathbb{D}$ and $\mathcal{C} (A)$ be the space of continuous functions on $A$ with its natural norm. Set:
\begin{displaymath} 
\tilde{d}_{n} (A) = \inf_{E} \raise - 2pt \hbox{$\bigg[$} \sup_{f \in B_{H^\infty}} {\rm dist} \, (f, E) \raise - 2pt \hbox{$\bigg]$} \, , 
\end{displaymath} 
where $E$ runs over all $(n - 1)$-dimensional subspaces of $\mathcal{C} (A)$ and 
${\rm dist} \, (f, E)  = \inf_{h \in E} \Vert f - h \Vert_{\mathcal{C}(A)}$. Then 
\begin{equation} 
\tilde{d}_{n} (A) \geq \alpha \, \e^{- n/\capa (A)}  
\end{equation} 
for some positive constant $\alpha$.
\end{theorem} 
\begin{proof} [Proof of Lemma~\ref{mino capa}] 
We apply Theorem~\ref{widom} to the compact set $A = u (r \, \overline{\D})$. 

Let  $E$ be an $(n - 1)$-dimensional subspace of $H^2 = H^2 (\D)$; it can be viewed as a subspace of 
${\mathcal C} (A)$, so, by Theorem~\ref{widom}, there exists $f \in H^\infty \subseteq H^2$ with $\| f \|_2 \leq \| f \|_\infty \leq 1$ such that:
\begin{displaymath} 
\qquad \| f - h \|_{{\mathcal C} (A)} \geq \alpha \, \Gamma^n \, , \quad \forall h \in E \, .
\end{displaymath} 
Then:
\begin{displaymath} 
\| v \, ( f \circ u - h \circ u)\|_{{\mathcal C} (r \T)} \geq \delta \, \| (f - h) \circ u \|_ {{\mathcal C} (r \T)} 
= \delta \, \| f - h\|_{{\mathcal C}(A)}  \geq \alpha \, \delta \, \Gamma^n \,.
\end{displaymath} 
But:
\begin{displaymath} 
\| v\, ( f \circ u - h \circ u)\|_{{\mathcal C} (r \T)} \leq \frac{1}{\sqrt{1 - r^2}} \, \| v \, ( f \circ u - h \circ u)\|_{H^2} \, ;
\end{displaymath} 
Hence:
\begin{displaymath} 
\| T f - T h\|_{H^2} \geq \alpha\, \sqrt{1 - r^2} \, \delta \, \Gamma^n \geq \alpha\, \sqrt{1 - r} \, \delta \, \Gamma^n\, .
\end{displaymath} 
Since $h$ is an arbitrary function of $E$, we get ($B_{H^2}$ being the unit ball of $H^2$):
\begin{displaymath} 
\inf_{\dim E < n} \raise -2pt \hbox{$\bigg[$} \sup_{f \in B_{H^2}} {\rm dist}\, \big(T f,  T (E) \big) \raise -2pt \hbox{$\bigg]$}
\geq \alpha\, \sqrt{1 - r} \, \delta \, \Gamma^n \, .
\end{displaymath} 
But the left-hand side is equal to the Kolmogorov number $d_n (T)$ of $T$ (see \cite{LQR}, Lemma~3.12), and, as recalled in \eqref{egalite}, in Hilbert 
spaces, the Kolmogorov numbers are equal to the approximation numbers; hence we obtain:
\begin{equation} 
\qquad a_n (T) \geq \alpha\, \sqrt{1 - r} \, \delta \, \Gamma^n \, , \quad n = 1, 2, \ldots \, ,
\end{equation} 
as announced.
\end{proof}

The next lemma shows that some Blaschke products are far away from $0$ on some circles centered at $0$.
\smallskip

We consider a \emph{strongly interpolating sequence} $(z_j)_{ j\geq 1}$ of $\D$ in the sense that, if $\varepsilon_j := 1 - |z_j|$, then:
\begin{equation} \label{regu}
\varepsilon_{j+1}\leq \sigma \,\varepsilon_j 
\end{equation}
and so $\varepsilon_j \leq \sigma^{j - 1} \varepsilon_1$, where $0 < \sigma  < 1$ is fixed. Equivalently, the sequence 
$(|z_j|)_{j \geq 1}$ is interpolating.  We consider the corresponding interpolating Blaschke product:
\begin{equation} \label{Blaschke}
B (z) = \prod_{j = 1}^\infty \frac{|z_j|}{z_j} \frac{z_{j} - z}{1 - z_{j} z} \, \cdot
\end{equation} 

The following lemma is probably well-known, but we could find no satisfactory reference (see yet \cite{HAY} for related estimates) and provide a simple proof.
\begin{lemma} \label{Lemma Blaschke}
Let $(z_j)_{j\geq 1}$ be a strongly interpolating sequence as in \eqref{regu} and $B$ the associated Blaschke product \eqref{Blaschke}. \par

Then there exists a sequence $r_l := 1 - \rho_l$ such that: 
\begin{equation} \label{encadrement rho}
C_1 \, \sigma^l\leq \rho_l \leq  C_2 \, \sigma^l \, , 
\end{equation} 
where $C_1$, $C_2$ are positive constants, and for which:
\begin{equation} \label{mino Blaschke} 
|z| = r_l \quad \Longrightarrow  \quad |B (z)|\geq \delta \, ,
\end{equation} 
where $\delta > 0$ does not depend on $l$.
\end{lemma} 
\begin{proof} 
Let us denote by $p_l$, $1\leq p_l \leq l$, the biggest integer such that $\varepsilon_{p_l} \geq \sigma^{l - 1}\varepsilon_1$. 

We separate two cases.
\smallskip

\noindent{\bf Case 1:} $\varepsilon_{p_l}\geq 2 \, \sigma^{l - 1}\varepsilon_1$. 

Then, we choose $\rho_l = \alpha \, \sigma^{l - 1}\varepsilon_1$ with $\alpha$ fixed, $1< \alpha < 2$.
Since $\rho (\xi, \zeta)\geq \rho(|\xi|,|\zeta|)$ for all $\xi, \zeta \in \D$ (recall that $\rho$ is the pseudo-hyperbolic distance on $\D$), we have the 
following lower bound for $|z|= r_l$:
\begin{displaymath} 
|B (z)| =\prod_{j = 1}^\infty \rho (z, z_j) \geq  \prod_{j = 1}^\infty \rho (r_l, |z_j|) 
= \prod_{j \leq p_l} \rho (r_l, |z_j|) \times \prod_{j > p_l} \rho (r_l, |z_j|) := P_1 \times P_2 \, ,
\end{displaymath} 
and we estimate $P_1$ and $P_2$ separately.

We first observe that $\dis \frac{\rho_l}{\varepsilon_{p_l}}\leq 
\frac{\alpha \, \sigma^{l - 1}\varepsilon_1}{2 \, \sigma^{l - 1}\varepsilon_1}\leq \frac{\alpha}{2}\,$,  
and then:  
\begin{displaymath} 
\frac{\rho_l}{\varepsilon_{j}} = \frac{\rho_l}{\varepsilon_{p_l}} \frac{\varepsilon_{p_l}}{\varepsilon_{j}}\leq  \frac{\alpha}{2} \, \sigma^{p_l - j}. 
\end{displaymath} 
The inequality $\rho (1 - u, 1 - v) \geq \frac{|u - v|}{(u + v)}$ for $0 < u, v \leq 1$ now gives us:
\begin{equation} \label{un} 
\quad \rho(r_l, |z_j|) \geq\frac{\varepsilon_j - \rho_{l}}{\varepsilon_j + \rho_{l}} 
=\frac{1 - \rho_l/\varepsilon_j}{1 + \rho_{l}/\varepsilon_j} 
\geq\frac{1 - (\alpha/2) \, \sigma^{p_l - j}}{1 + (\alpha/2) \, \sigma^{p_{l}-j}} \, \raise 1 pt \hbox{,} \quad \text{for } j\leq p_l \, ,
\end{equation}
and:
\begin{equation} \label{uun} 
P_1 \geq \prod_{k = 0}^\infty \bigg( \frac{1 - (\alpha/2) \, \sigma^{k}}{1 + (\alpha/2) \, \sigma^{k}} \bigg) \, \cdot 
\end{equation}
Similarly: 
\begin{displaymath} 
\frac{\varepsilon_{p_{l} + 1}}{\rho_{l}}\leq \frac{\sigma^{l - 1}\varepsilon_1}{\alpha \, \sigma^{l - 1} \varepsilon_1} \leq \frac{1}{\alpha} 
\end{displaymath} 
and:
\begin{displaymath} 
\qquad \quad \frac{\varepsilon_{j}}{\rho_{l}}\leq \frac{1}{\alpha} \, \sigma^{j - p_{l} - 1} \quad \text{for }   j > p_l \, ;
\end{displaymath} 
so that:
\begin{equation} \label{deux} 
\quad \rho (r_l, |z_j|) \geq \frac{\rho_l - \varepsilon_j}{\rho_l + \varepsilon_j} 
= \frac{1 - \varepsilon_j/\rho_l}{1+\varepsilon_j/\rho_l} 
\geq \frac{1 - \alpha^{- 1} \sigma^{j - p_{l} - 1}}{1 + \alpha^{-1} \sigma^{j - p_{l} - 1}} \, \raise 1 pt \hbox{,} \quad \text{for } j > p_l \, ,
\end{equation}
and
\begin{equation} \label{deeux} 
P_2 \geq \prod_{k = 0}^\infty \bigg( \frac{1 - \alpha^{- 1} \sigma^{k}}{1 + \alpha^{-1} \sigma^{k}} \bigg) \, \cdot 
\end{equation}
Finally, the condition of lower and upper bound for $\rho_l$ is fulfilled by construction.
\smallskip\goodbreak
 
\noindent{\bf Case 2:} $\varepsilon_{p_l} \leq 2 \, \sigma^{l - 1}\varepsilon_1$.

Then, we choose $\rho_l = a \, \varepsilon_{p_l}$ with $\sigma < a < 1$ fixed. Computations exactly similar to those of Case~1 give us:
\begin{equation} \label{trois} 
|B (z)| 
\geq \prod_{k = 0}^\infty \bigg(\frac{1 - a \, \sigma^{k}}{1 + a \, \sigma^{k}}\bigg) \times 
\prod_{k = 0}^\infty \bigg( \frac{1 - a^{-1} \sigma^{k}}{1 + a^{- 1} \sigma^{k}}\bigg) =: \delta > 0 \, , \quad \text{for } |z| = r_l  \, .  
\end{equation}
Moreover, in this case:
\begin{displaymath} 
a \, \sigma^{l - 1}\varepsilon_1 \leq \rho_l \leq 2 \, a \, \sigma^{l - 1}\varepsilon_1 \, ,
\end{displaymath} 
and the proof is ended.
\end{proof}

\medskip
Now, we have the following estimation. \goodbreak
\begin{theorem} \label{Theo general minoration}
Let $\phi, \psi \colon \D \to \D$ be two non-constant analytic self-maps and $\Phi (z_1, z_2) = \big( \phi (z_1), \psi (z_1) \, h (z_2)\big)$, where $h$ is 
inner.  


Let $(r_l)_{l \geq 1}$ be an increasing sequence of positive numbers with limit $1$ such that:
\begin{displaymath} 
\inf_{|z| = r_l} |\psi (z) | \geq \delta_l > 0  \, ,
\end{displaymath} 
with $\delta_l \leq \e^{- 1 / \capa (A_l)}$, where $A_l = \phi \big( r_l  \overline{\D} \big)$. \par

Then the approximation numbers $a_N (C_\Phi)$,  $N \geq 1$, of the composition operator $C_\Phi \colon H^2 (\D^2) \to H^2 (\D^2)$ satisfy:
\begin{equation} \label{estim inf}
a_N (C_\Phi) \gtrsim 
\sup_{l \geq 1} \Big[ \sqrt{1 - r_l} \, \exp \big( - 8 \, \sqrt{N} \, \sqrt{\log ( 1 / \delta_l)}  \, \sqrt {\log (1 / \Gamma_l} \, \, \big) \Big] \, ,
\end{equation} 
where:
\begin{equation} 
\Gamma_l = \e^{- 1 / \capa (A_l)} \, .
\end{equation} 
\end{theorem} 
\begin{proof}
Since $h$ is inner, the sequence $(h^k)_{k \geq 0}$ is orthonormal in $H^2$ and hence $a_n (C_\Phi) = a_n (T)$ for all $n \geq 1$, where 
$T = \bigoplus_{k = 0}^\infty T_k$ and $T_k = M_{\psi^k} C_\phi$. Then Lemma~\ref{mino capa} gives:
\begin{equation} 
a_n (T_k) \gtrsim \sqrt{1 - r_l} \, \delta_l^k \Gamma_l^n
\end{equation} 
for all $n \geq 1$ and all $k \geq 0$.

Let now: 
\begin{equation} 
p_l =  \bigg[ \frac{\log (1 /\delta_l )}{\log (1/\Gamma_l)} \bigg]  \,,
\end{equation} 
where $[\, . \,]$ stands for the integer part, and: 
\begin{equation} 
\hskip 90 pt n_k = p_l  k \, , \hskip 45 pt \text{for} \quad  k = 1, \ldots, K \,. \qquad
\end{equation} 
By Lemma~\ref{lemme utile}, applied with $m_k = k$ (i.e. to $H_1, \ldots, H_K$), we have, if $N = n_1 + \cdots + n_K$:
\begin{displaymath} 
a_N (T) \geq \inf_{1 \leq k \leq K} \alpha\, \sqrt{1 - r_l} \, \delta_l^k \, \Gamma_l^n = \alpha \, \sqrt{1 - r_l} \, \delta_l^K \, \Gamma_l^{n_K} \, .
\end{displaymath} 
But, since $p_l \leq \log (1 / \delta_l) / \log (1 / \Gamma_l)$:
\begin{align*}
\delta_l^K \, \Gamma_l^{n_K} 
& = \exp\big[ - \big( K \log (1 / \delta_l) + p_l K \log(1 / \Gamma_l) \big) \big] 
\geq \exp [ - 2 K \log (1 / \delta_l) ] \, .
\end{align*}
Since:
\begin{displaymath} 
N = p_l \frac{K (K +1)}{2} \geq p_l \frac{K^2}{4} \geq \frac{K^2}{16} \, \frac{\log (1 / \delta_l)}{\log (1 / \Gamma_l)} \, \raise 1,5 pt \hbox{,}
\end{displaymath} 
we get:
\begin{displaymath} 
\delta_l^K \, \Gamma_l^{n_K} \geq \exp \big[ - 8 \, \sqrt{N} \, \sqrt{\log (1 / \delta_l)} \sqrt{ \log (1 / \Gamma_l)} \big] \, ,
\end{displaymath} 
and the result ensues.
\end{proof}
\medskip

\noindent{\bf Example 1.} We take $\phi = \lambda_\theta$, a lens map, and $\psi = B$, a Blaschke product associated to a strongly regular sequence, as 
defined in \eqref{Blaschke}; then we get:
\begin{theorem} \label{Prop mino lens}
Let $\Phi \colon \D^2 \to \D^2$ be defined by:
\begin{displaymath} 
\Phi (z_1, z_2) = \big(\lambda_\theta (z_1) , c\, B (z_1)\, h (z_2) \big) \, , 
\end{displaymath} 
where $B$ is a Blaschke product as in \eqref{Blaschke}, $0 < c < 1$, and $h$ is an arbitrary inner function, we have, for some positive constant $b$, for all 
$N \geq 1 $:
\begin{equation} 
\qquad a_N (C_\Phi) \gtrsim \exp ( - b \, N^{1/3}) = \exp ( - b \, \sqrt{N} / N^{1/6} ) \, .
\end{equation} 
In particular $\beta_2 (C_\Phi) = \beta_2^\pm (C_\Phi) = 1$. 
\end{theorem} 

\noindent{\bf Remark.} We saw in Theorem~\ref{majo lens} that this is the exact size, since we have: $a_N (C_\Phi) \lesssim \e^{- \beta \, N^{1/3}}$. 
\begin{proof}
By Lemma~\ref{Lemma Blaschke}, there is a sequence of numbers $r_l \approx \sigma^l$ such that $|B (z) | \geq \delta$ for $|z| = r_l$, where $\delta$ is a 
positive constant (depending on $\sigma$). Since $\lambda_\theta (0) = 0$, we have:
\begin{displaymath} 
{\rm diam}_\rho (A_l) \geq \lambda_\theta (r_l) \gtrsim 1 - (1 - r_l)^\theta \, ;
\end{displaymath} 
hence, by \cite{LQR}, Theorem~3.13, we have:
\begin{displaymath} 
\capa (A_l) \gtrsim \log \frac{1}{1 - r_l} \gtrsim l \, ,
\end{displaymath} 
or, equivalently: $\Gamma_l \geq \e^{ - b / l}$, some some $b > 0$. Then \eqref{estim inf} gives, for all $l \geq 1$ (with another $b$):
\begin{displaymath} 
a_N (C_\Phi) \gtrsim \exp \bigg[ - b \bigg( l + \frac{\sqrt{N} }{\sqrt l} \bigg) \bigg] \, .
\end{displaymath} 
Taking $l = N^{1/3}$, we get the result. 
\end{proof}
\medskip

\noindent {\bf Example 2.} By taking the cusp instead of a lens map, we obtain a better result, close to the extremal one. 
\begin{theorem} 
Let $\Phi (z_1, z_2) = \big( \chi (z_1), c\, B (z_1) \, h (z_2) \big)$, where $\chi$ is the cusp map, $B$ a Blaschke product as in \eqref{Blaschke}, $0 < c < 1$, 
and $h$ an arbitrary inner function. Then, for all $N \geq 1$:
\begin{displaymath} 
a_N (C_\Phi) \gtrsim \e^{ - b \, \sqrt{N} / \sqrt{\log N }} \, .
\end{displaymath} 
In particular $\beta_2 (C_\Phi) = 1$.
\end{theorem} 
\noindent{\bf Remark.} We saw in Theorem~\ref{majo cusp} that this is the exact size, since we have: 
$a_N (C_\phi) \lesssim \e^{ - \beta \sqrt{N / \log N}}$.  
\begin{proof}
The proof is the same as that of Proposition~\ref{Prop mino lens}, except that, for the cusp map, we have (note that $\chi (0) = 0$): 
\begin{displaymath} 
{\rm diam}_\rho (A_l) \geq \chi (r_l) \, . 
\end{displaymath} 
But when $r$ goes to $1$:
\begin{displaymath} 
1 - \chi (r) \sim \frac{\pi \, (\sqrt{2} - 1)}{2} \, \frac{1}{\log \big( 1 / (1 - r) \big)} 
\end{displaymath} 
(see \cite{LIQURO-Estimates}, Lemma~4.2). Hence, by \cite{LQR}, Theorem~3.13, again, we have:
\begin{displaymath} 
\capa (A_l) \gtrsim \log \big( \log \big( 1 / (1 - r_l) \big) \, ,
\end{displaymath} 
so $\Gamma_l \geq \e^{ - b / \log l}$. Then, \eqref{estim inf} gives (with another $b$):
\begin{displaymath} 
a_N (C_\Phi) \gtrsim \exp \bigg[ - b \bigg( l + \frac{\sqrt{N}}{\sqrt{\log l}} \bigg) \bigg] \, .
\end{displaymath} 
In taking $l = \sqrt{N / \log N}$, we get the announced result.
\end{proof}
%
\subsection{Upper bounds}

All previous results point in the direction that, if $\Vert \Phi \Vert_\infty = 1$, then however small $a_{n}(C_\Phi)$ is, it will always be larger than 
$\alpha \, \e^{- \beta \varepsilon_n \sqrt n}$ with $\varepsilon_n \to 0^{+}$, as this is the case in dimension one (with $n$ instead of $\sqrt n$). But 
Theorem~\ref{chobou} to follow shows that we cannot hope, in full generality, to get the same result in dimension $d \geq 2$, and that other phenomena await 
to be understood. Here is our main result. It shows that, even for a truly $2$-dimensional symbol $\Phi$, we can have $\Vert\Phi \Vert_\infty = 1$ and 
nevertheless $\beta_{2}^{+} (C_\Phi) < 1$, in contrast to the $1$-dimensional case where \eqref{equiv-dim1} holds.

\begin{theorem} \label{chobou} 
There exist a map $\Phi \colon \D^2\to \D^2$ such that: 
\par\smallskip

$1)$  the composition operator $C_\Phi \colon H^{2}(\D^2)\to  H^{2}(\D^2)$ is bounded and compact;
\par\smallskip

$2)$ we have $\Vert\Phi \Vert_\infty = 1$ and $\Phi$ is truly $2$-dimensional, so that  $\beta_{2}^{-}(C_\Phi) > 0$;
\par\smallskip

$3)$ the singular numbers  satisfy  $a_{n}(C_\Phi) \leq \alpha \, \e^{-\beta \,\sqrt{n}}$ for some positive constants $\alpha, \beta$; in particular 
$\beta_{2}^{+} (C_\Phi) < 1$.
\end{theorem}
\begin{proof}
Let $0 < \theta < 1$ be fixed, and $\lambda_\theta$ be the corresponding lens map. We set:
\begin{displaymath} 
\left\{
\begin{array} {lcll}
& \phi  & =  & \dis \frac{1 + \lambda_\theta}{2} \smallskip \\ 
& w (z) & = & \dis \exp \bigg[ - \bigg( \frac{1+ z}{1 - z}\bigg)^\theta\, \bigg] \smallskip\\  
& \psi & = & w\circ \phi \, .
\end{array}
\right.
\end{displaymath}
Note that $\|\phi \|_\infty = 1$.
\smallskip

Setting $\delta = \cos (\theta\pi/2) > 0$, we have for $z\in \D$:
\begin{equation} \label{choi0}
|1 - \phi (z)|= \frac{1}{2} \, |1 - \lambda_{\theta}(z)| = \bigg|\frac{(1 - z)^\theta}{(1 - z)^\theta + (1 + z)^\theta}\bigg|
\leq \frac{|1 - z|^\theta}{\delta} \, \cdot
\end{equation}
Indeed, the argument $\alpha$ of $(1 \pm z)^\theta$ satisfies $|\alpha| \leq \theta \pi/2$ for $z\in \D$,  and we get:
\begin{displaymath} 
|(1 - z)^\theta + (1 + z)^\theta|\geq \Re [(1 - z)^\theta + (1+ z)^\theta] \geq \delta(|1 + z|^\theta + |1 - z|^\theta) \geq \delta \, .
\end{displaymath} 
We also see that $\phi (\D)$ touches the boundary $\partial \D$ only at $1$ in a non-tangential way, meaning that for some constant $C>1$:
\begin{displaymath} 
\qquad \quad 1 - |\phi (z)| \geq \frac{1}{C} \, |1 - \phi (z)| \, , \quad  \forall z \in \D \, .
\end{displaymath} 

Now, we have the following two inequalities:
\begin{align}
& \Re z \geq 0 \quad \Longrightarrow  \quad |w (z)|\leq \exp \bigg( - \frac{\delta}{|1 - z|^\theta} \bigg)  \label{choi1} \\
& z \in \D \quad \Longrightarrow \quad |\psi (z)| \leq \exp\bigg( - \frac{\delta^{2}}{|1 - z|^{\theta^2}}\bigg) \, . \label{choi2}
\end{align}
Indeed, with $S (z)  = \big( \frac{1+ z}{1 - z}\big)^\theta$, we have $\Re\, S (z) \geq \delta \, |S (z)|\geq \delta |1 - z|^{- \theta}$ when $\Re z\geq 0$, giving \eqref{choi1}, and  \eqref{choi0} and \eqref{choi1} imply, since $\Re \phi (z)\geq 0$:
\begin{displaymath} 
|\psi (z)| = |w \big( \phi (z) \big)|\leq \exp \bigg( - \frac{\delta}{|1 - \phi (z)|^\theta} \bigg) \leq \exp \bigg( - \frac{\delta^{2}}{|1 - z|^{\theta^2}}\bigg) \, \cdot
\end{displaymath} 
\medskip 

We now set:
\begin{equation} 
\Phi (z_1, z_2) = \big(\phi (z_1), \psi (z_1) \, h (z_2) \big) \, ,
\end{equation} 
with $h$ a Rudin function. 
\smallskip

Observe that $\phi \in A(\D)$ and $\psi = w\circ \phi \in A(\D)$ as well ($w \in A(\D)$ with $w (1) = 0$; this is due to the presence of the parameter 
$\theta < 1$). hence if we take for $h$ a finite Blaschke product, the two components  of $\Phi$ are in the bidisk algebra $A (\D^2)$.\par
\smallskip

We have $\| \psi \|_\infty := \rho < 1$. In fact, for $\Re u \geq 0$, we have:
\begin{displaymath} 
\bigg| \frac{1 + u}{1 - u} \bigg| \geq 2^{- \theta} | 1 + u|^\theta \geq 2^{- \theta} (1 + \Re u)^\theta \geq 2^{- \theta} \, ,
\end{displaymath} 
hence:
\begin{displaymath} 
\Re \bigg[ \bigg( \frac{1 + u}{1 - u} \bigg)^\theta \, \bigg] \geq \bigg( \cos \frac{\theta \pi}{2} \bigg) \, \bigg| \frac{1 + u}{1 - u} \bigg|^\theta 
\geq \bigg( \cos \frac{\theta \pi}{2} \bigg) \, 2^{- \theta} = \delta \, 2^{- \theta} \, ,
\end{displaymath} 
and $\| w \circ \phi \|_\infty \leq \e^{2^{- \theta} \delta}$.\par
\smallskip

Now, $1)$ follows from the orthogonal model presented in Section~\ref{general facts}, because $\| \psi \|_\infty < 1$. \par

The assertion $2)$ follows from \cite{BLQR}, Theorem~3.1, since $\| \phi \|_\infty = 1$.\par

We now prove $3)$. 

As observed, $C_\Phi$ can be viewed as a direct sum $T = \bigoplus_{k = 0}^\infty T_k$ acting on a Hilbertian sum $H = \bigoplus_{k = 0}^\infty H_k$, 
where $T_k$ acts on a copy $H_k$ of $H^{2}(\D)$ with:
\begin{displaymath} 
T_k = M_{\psi^{k}} C_\phi \, . 
\end{displaymath} 
We fix the positive integer $n$.  The rest of the proof will consist of three lemmas.

\begin{lemma} \label{rest1}  
We have $\Vert T_k\Vert \leq 2 \, \rho^{- k} \leq 2\, \rho^{ - n}$ for $k > n$. 
\end{lemma}
\begin{proof}
Indeed, since $\phi (0) = 1/2$, we know that $\Vert C_{\phi}\Vert \leq \sqrt{\frac{1 + \phi(0)}{1 - \phi(0)}}=\sqrt{3}\leq 2$, 
so that $\Vert T_k\Vert\leq \Vert \psi^k\Vert_\infty \Vert C_\phi\Vert\leq \rho^{- k}\times 2$.
\end{proof}
\begin{lemma} \label{rest2} 
Set $b = a /\delta^2 $ where $a > 0$ is given by $\e^{- a}=4C /\sqrt{16 C^2+1}$ and $C$ is as in \eqref{veritas}. Let $m_k$ be the smallest integer such that 
$k \, \delta^{2} 2^{m_{k} \theta^2} \geq a n$;  namely:
\begin{equation} \label{smal} 
m_k =\bigg[\frac{\log(b\,n/k)}{\theta^{2}\log 2} \bigg] + 1 \, ,
\end{equation} 
where $[\, . \, ]$ stands for the integer part. Then, with $a'=\min(\log 2,a)$: 
\begin{displaymath} 
a_{n m_k + 1} (T_k) \lesssim \e^{- a'n} \, .
\end{displaymath} 
\end{lemma}
\begin{proof}
This follows from Theorem~\ref{special} applied with $w = \psi^{k}$,  $R = k \, \delta^2$ and $\theta$ changed into $\theta^2$.  This is possible thanks to 
\eqref{choi2} and to Lemma~\ref{rest1}. Moreover we have adjusted $m_k$ so as to make the two terms in Theorem~\ref{special} of the same order.  
\end{proof}
\begin{lemma} \label{rest3} 
The dimension $d :=\sum_{k = 0}^n n \, m_{k}$ satisfies, for some positive constant $\alpha$:
\begin{displaymath} 
d \leq \alpha \, n^2 \, .
\end{displaymath} 
\end{lemma} 
\begin{proof}
Indeed, it is well-known that: 
\begin{displaymath} 
\sum_{k=1}^{n}\log k = n\log n - n + {\rm O}\, (\log n) \, ,
\end{displaymath} 
and, in view of \eqref{smal}, we have $m_k \leq \alpha'_\theta \log (b \, n /k) \leq \alpha''_\theta (\log n - \log k)$; hence:
\begin{displaymath} 
\sum_{k = 1}^n m_k \leq \alpha''_\theta \big[ n \log n - \big(n \log n - n + {\rm O}\, (\log n) \big) \big] 
= \alpha''_\theta \, n + {\rm O}\, (\log n) \, ,
\end{displaymath} 
and we get $d \leq \alpha''_\theta \, n^2 + {\rm O}\, (n \log n) \leq \alpha_\theta \, n^2$.
\smallskip

Alternatively, we could have used a Riemann sum for the function $\log (1/x)$ on $(0, 1]$.
\end{proof}

Finally, putting things together and using as well Proposition~\ref{majo} with $K = n$ and $n_k = n m_k + 1$ so that 
$(\sum_{k = 0}^n n_k) - n = (\sum_{k = 0}^n n \, m_k) + 1 = d + 1$, we get ignoring once more multiplicative constants: 
\begin{displaymath} 
a_{n^2} (T) \lesssim a_{d} (T) \leq \alpha \, \e^{- \beta n} 
\end{displaymath} 
with  positive constants $\alpha$, $\beta$. 
This  ends the proof of Theorem~\ref{chobou}.
\end{proof}
%

\section{Monge-Amp\`ere capacity and applications} \label{sec: capacity}

\subsection{Definition} 

Let $K$ be a compact subset of $\D^m$ (in this section, for notational reasons, we denote the dimension by $m$ instead of $d$). The Monge-Amp\`ere 
capacity of $K$ has been defined by Bedford and Taylor (\cite{BT}; see also \cite{KOSK}, \S~5 or \cite{Klimek}, Chapter~1) as:
\begin{displaymath} 
\capam (K) = \sup \bigg\{ \int_K (dd^c u)^m \, ; \ u \in PSH \text{ and } 0 \leq u \leq 1 \bigg\} \, ,
\end{displaymath} 
where $PSH$ is the set of plurisubharmonic functions on $\D^m$, $dd^c = 2 i \partial \bar\partial$, and $(dd^c)^m = dd^c \wedge \cdots \wedge dd^c$ 
($m$ times). When $u \in PSH \cap\, {\cal C}^2 (\D^m)$, we have:
\begin{displaymath} 
(dd^c u)^m = 4^m m! \det \bigg( \frac{\partial^2 u}{\partial z_j \partial \bar{z}_k} \bigg)\, dV (z) \, , 
\end{displaymath} 
where $dV (z) = (i/2)^m dz_1 \wedge d\bar{z}_1 \wedge \cdots \wedge dz_m \wedge d\bar{z}_m$ is the usual volume in $\C^m$. A more convenient 
formula (because $\D^m$ is bounded and hyperconvex: see \cite{Klimek}, p.~80, for the definition) is:
\begin{displaymath} 
\capam (K) = \int_K (dd^c u_K^\ast)^m \, ,
\end{displaymath} 
where $u_K^\ast$ is called the \emph{extremal function of $K$} and is the upper semi-continuous regularization of:
\begin{displaymath} 
u_K = \sup\{ v \in PSH \, ; \ v \leq 0 \text{ and } v \leq - 1 \text{ on } K\} \, ,
\end{displaymath} 
but we will not need that. 
\par\medskip

As in \cite{ZAK}, we set:
\begin{equation} 
\tau_m (K) = \frac{1}{(2 \pi)^m} \, \capam (K) \ .
\end{equation} 
For $m = 1$, $\tau (K) := \tau_1 (K)$ is equal to the Green capacity $\capa (K)$ of $K$ with respect to $\D$, with the definition used in \cite{LQR} 
(see \cite{KOSK}, Theorem~8.1, where a factor $2 \pi$ is introduced). 
\smallskip

We further set:
\begin{equation} \label{set} 
\Gamma_{m} (K) = \exp \bigg[ - \bigg( \frac{m!}{\tau_m (K)}\bigg)^{1/m}\bigg] \, \cdot 
\end{equation}

We proved in \cite{LQR} that, for $m = 1$, and $\varphi \colon \D \to r\D$, with $0 < r <1$, we have:
\begin{equation} \label{spectral}
\beta_1 (C_\varphi) = \Gamma_1 \big( \overline{\varphi (\D)} \big) \, .
\end{equation} 

The goal of this section is to see that Theorem~\ref{chobou} shows that this no longer holds for $m = 2$.

\subsection{A seminal example}

In one variable, our initial motivation had been the simple-minded example $\varphi (z) = r z$, $0 < r < 1$, for which $C_{\varphi}(z^n) = r^n z^n$, implying 
$a_{n}(C_\varphi) = r^{n - 1}$ and $\beta_1 (C_\varphi) = r$. If $K = \overline{\varphi (\D)} =\overline{D} (0, r)$, we have  
$\capa (K) = \frac{1}{\log 1/r}$ and $\Gamma_1 (K) = r$, so that $\beta_1 (C_\varphi) = \Gamma_1 (K)$. Let us examine the multivariate  example 
(where $0< r_j < 1$): 
\begin{displaymath} 
\Phi (z_1, z_2, \ldots, z_m) = (r_1 z_1, r_2 z_2, \ldots, r_m z_m).
\end{displaymath} 
If $K = \overline{\Phi (\D^m)}$, we have $K = \prod_{k = 1}^m \overline{D} (0, r_k)$, and hence (\cite{Blocki}, Theorem~3):
\begin{equation} \label{thus} 
\tau_m (K) =\prod_{k = 1}^m \frac{1}{\log (1/r_k)} \, \cdot
\end{equation}
On the other hand, 
$C_{\Phi} (z_{1}^{n _1} z_{2}^{n_2}\cdots z_{m}^{n_m}) 
= r_{1}^{n_1} r_{2}^{n_2}\cdots r_{m}^{n_m} \, z_{1}^{n_1} z_{2}^{n_2}\cdots z_{m}^{n_m}$ so that  the sequence $(a_n)_n$ of approximation 
numbers of $C_\Phi$ is the non-increasing rearrangement of the numbers $r_{1}^{n_1} r_{2}^{n_2}\cdots r_{m}^{n_m}$. It is convenient to state the 
following simple lemma.
\begin{lemma} 
Let $\lambda_1, \ldots, \lambda_m$ be positive numbers. Let $N_A$ be the number of $m$-tuples of non-negative integers $(n_1, \ldots, n_m)$ such that 
$\sum_{k = 1}^m \lambda_k n_k\leq A$. Then, as $A\to \infty$:
\begin{displaymath} 
N_A \sim \frac{A^m}{(\lambda_1\cdots \lambda_m)\, m!}  \, \cdot 
\end{displaymath} 
\end{lemma}

Indeed, just apply Karamata's tauberian theorem (see \cite{KOR} p.~30) to the generalized Dirichlet series:
\begin{displaymath} 
S (\varepsilon) := \prod_{k=1}^m \frac{1}{1 - \e^{- \lambda_k \varepsilon}} 
\quad = \sum_{n_1, \ldots, n_m \geq 0} \e^{ - (\sum_{k = 1}^m \lambda_{k} n_k) \, \varepsilon} \, ;
\end{displaymath} 
we have $S (\varepsilon) \sim \frac{\varepsilon^{- m}}{(\lambda_1 \cdots \lambda_m)}$ as $\varepsilon \to 0^{+}$.
\smallskip

Let now $N$ be a positive integer and   $\varepsilon = a_N$. Setting $\lambda_k = \log (1/r_k)$ and $A = \log (1/\varepsilon)$, we see that $N$ is the 
number of $m$-tuples $(n_1,\ldots, n_m)$ of non-negative integers such that $r_{1}^{n_1} r_{2}^{n_2}\cdots r_{m}^{n_m}\geq \varepsilon$, i.e. such that  
$\sum_{k = 1}^m \lambda_k n_k \leq A$. This number $N$ is hence nothing but the number $N_A$ of the previous lemma, so that:
\begin{displaymath} 
N \sim \frac{A^m}{(\lambda_1 \cdots \lambda_m) \, m!}  \, \cdot
\end{displaymath} 
Inverting this formula, we get:
\begin{displaymath} 
a_N (C_\Phi) = \exp \big[ - (1 + {\rm o}\, (1)) \, (m! (\lambda_1 \lambda_2\cdots \lambda_m)\, N)^{1/m}\big]
\end{displaymath} 
and:
\begin{displaymath} 
\beta_{m} (C_\Phi) = \exp\big[ - (m! \lambda_1 \lambda_2\cdots \lambda_m)^{1/m} \big] = \Gamma_{m} (K) \, ,
\end{displaymath} 
in view of \eqref{set} and \eqref{thus}. 
\medskip

On the view of the simple-minded previous example, the extension of the spectral radius formula \eqref{spectral} to the multivariate case holds, and it is 
tempting to conjecture that this is a general phenomenon as in dimension one, all the more as the following extension of Widom's theorem was proved by 
Zakharyuta, based on the solution by S.~Nivoche of Zakharyuta's conjecture (\cite{Nivoche}); see also \cite{ZAK}, Theorem~5.4. A compact subset 
$K$ of $\D^m$ is said to be \emph{regular} if its extremal function $u_K^\ast$ is continuous on $\D^m$.
\begin{theorem} [\cite{ZAK}, Theorem~5.6] \label{Zakha-exact}
Let $K$ be a regular compact subset of $\D^m$ and $J \colon H^\infty (\D^m) \to {\cal C} (K)$ the canonical injection; then the Kolmogorov numbers 
$d_n (J)$ satisfy:
\begin{equation} 
\lim_{n \to \infty} \big[d_n (J) \big]^{1/n^{1 /m}} = \exp \bigg[ - \bigg(\frac{m!}{\tau_m (K)} \bigg)^{1/m} \bigg] \, \cdot
\end{equation} 
\end{theorem} 

Note that the right side is nothing but $\Gamma_m (K)$. 
\smallskip

We will see consequences of this result in a forthcoming paper (\cite{LQR-Pluricap}).

\subsection{Upper bound} 

For the upper bound, the situation behaves better, as stated in the following theorem.
\begin{theorem} [\cite{ZAK}, Proposition~6.1] \label{Zakha-upper}
Let $K$ be a compact subset of $\D^m$ with non-void interior. Then:
\begin{equation} 
\limsup_{n \to \infty} \big[ d_n (J) \big]^{1/n^{1/m}} \leq \exp \bigg[ - \bigg( \frac{m!}{\tau_m (K)} \bigg)^{1/m} \bigg] \, .
\end{equation} 
\end{theorem} 

Note that $(K, \D^m)$ is a condenser since $K$ has non-void interior. We deduce the following upper bound.
\begin{theorem} \label{extension} 
Let $\Phi$ be an analytic self-map of  $\D^m$ with $\Vert \Phi \Vert_\infty = \rho < 1$, thus inducing a compact composition operator on 
$H^{2}(\D^m)$. Then we have:
\begin{displaymath} 
\beta_{m}^{+} (C_\Phi) \leq \Gamma_{m} \big( \overline{\Phi (\D^m)} \big) \, .
\end{displaymath} 
\end{theorem}
\begin{proof}
This proof provides in particular a simplification of that given in \cite{LQR} in dimension $m = 1$.
\smallskip

Changing $n$ into $n^m$, Theorem~\ref{Zakha-upper} means that for every $\varepsilon > 0$, there exists an $(n^{m} - 1)$-dimensional subspace $V$ of 
$\mathcal{C}(K)$ such that, for any $g \in H^{\infty} (\D^m)$, there exists $h \in V$ such that:
\begin{equation} \label{demi}
\Vert g - h\Vert_{\mathcal{C}(K)} \leq C_\eps (1 + \varepsilon)^n \big[ \Gamma_{m}(K) \big]^n \Vert g \Vert_\infty \, .  
\end{equation} 

Let $l$ be an integer to be adjusted later, and  $f (z) = \sum_{\alpha} b_\alpha z^\alpha \in B_{H^2}$, as well as 
$g (z) = \sum_{|\alpha|\leq l} b_\alpha z^\alpha$. We first note that (with $M_m$ depending only on $m$ and $\rho$, and since the number of $\alpha$'s 
such that $|\alpha|\leq p$ is ${\rm O}\, (p^m)$): 
\begin{displaymath} 
\sum_{|\alpha|> l} \rho^{2|\alpha|} \leq M_m \sum_{p >  l} p^{m} \, \rho^{2 p} \leq M_m l^{m} \, \frac{\rho^{2 l}}{(1 - \rho^2)^{m + 1}} \, \cdot
\end{displaymath} 
We next observe that, by the Cauchy-Schwarz and Parseval inequalities:
\begin{equation}\label{une} 
\Vert g \Vert_\infty \leq M_m \, l^{m/2} \, ,
\end{equation}
and 
\begin{equation} \label{une-et-demi}
\qquad  |f (z) - g (z)| \leq M_m \, l^{m/2}\frac{|z|_{\infty}^{l} \ \quad}{(1 - |z|_{\infty}^2)^{(m + 1)/2}} \, \raise 1 pt \hbox{,} 
\qquad  \forall z \in \D^m \, .
\end{equation} 
where $|z|_\infty :=\max_{j \leq m} |z_j|$ if $z = (z_1, \ldots, z_m)$.
\smallskip

The subspace $F$ formed by functions $v \circ \Phi$, for $v\in V$, can be viewed as a subspace of $L^{\infty}(\T^m) \subseteq L^{2}(\T^m)$ with respect 
to the Haar measure of $\T^m$, the distinguished boundary of $\D^m$ (indeed, we can write $(v \circ \Phi)^\ast = v \circ \Phi^\ast$, where $\Phi^\ast$ 
denotes the almost everywhere existing radial limits of $\Phi (r z)$, which belong to $K$). Let finally $E = P (F) \subseteq H^2 (\D^m)$ where 
$P \colon L^{2}(\T^m)\to H^2 (\T^m) = H^2 (\D^m)$ is the orthogonal projection. This is a subspace of $H^2$ with dimension $< n^m$.  Set temporarily 
$\eta = C_\eps (1 + \varepsilon)^n \big[ \Gamma_{m}(K) \big]^n$. It follows from \eqref{demi} and \eqref{une} that, for some $h \in V$:
\begin{displaymath} 
\Vert g - h \Vert_{\mathcal{C}(K)} \leq \eta \, \Vert g \Vert_\infty \leq \eta \, M_m \, l^{m/2} 
\end{displaymath} 
and hence:
\begin{displaymath} 
\Vert g \circ \Phi - h \circ \Phi \Vert_{2} \leq \Vert g \circ \Phi - h \circ \Phi \Vert_{\infty} \leq \eta \, M_m \, l^{m/2} \, , 
\end{displaymath} 
implying by orthogonal projection:
\begin{displaymath} 
{\rm dist}\, (C_\Phi g, E) \leq \Vert g \circ \Phi - P (h \circ \Phi) \Vert_{2} \leq \eta \, M_m \, l^{m/2} \, .
\end{displaymath} 
Now, since $C_\Phi f (z) - C_\Phi g (z) = f \big( \Phi (z) \big) - g \big( \Phi (z) \big)$, \eqref{une-et-demi} gives:  
\begin{displaymath}  
\Vert C_\Phi f - C_\Phi g \Vert_2 \leq \Vert C_\Phi f - C_\Phi g \Vert_\infty \leq M_m \, l^{m/2} \, \frac{\rho^l}{(1 - \rho^2)^{(m + 1)/2}} 
\end{displaymath} 
and hence: 
\begin{displaymath} 
{\rm dist}\, (C_\Phi f, E) 
\leq M_m \, l^{m/2} \,\bigg( \frac{\rho^l}{(1 - \rho^2)^{(m + 1)/2}} + C_\eps (1 + \eps)^n \big[ \Gamma_{m} (K) \big]^n \bigg) \, .
\end{displaymath}
It ensues, since $a_N (C_\Phi) = d_N (C_\Phi)$, that:
\begin{displaymath} 
\big[ a_{n^m} (C_\Phi) \big]^{1/n}
\leq  (M_m \, l^{m/2})^{1/n} \, \bigg[ \frac{\rho^{l/n}}{(1 - \rho^2)^{(m + 1)/2n}} + C_\eps^{1/n} (1 + \eps) \,  \Gamma_{m}(K) \bigg] \, .
\end{displaymath} 
Taking now for $l$ the integer part of $n \log n$, and passing to the upper limit as $n \to \infty$, we obtain (since $l / n \to \infty$ and $(\log l)/n\to 0$):
\begin{displaymath} 
\beta_{m}^{+} (C_\Phi) \leq (1 + \eps) \,\Gamma_{m} (K) \, ,
\end{displaymath} 
and Theorem~\ref{extension} follows.
\end{proof}
%

\bigskip
\noindent{\bf Acknowledgements:} The two first-named authors would like to thank the colleagues of the University of Sevilla for their kind hospitality, which 
allowed a pleasant and useful stay, during which this collaboration was initiated. They also thank E.~Fricain, S.~Nivoche, J.~Ortega-Cerd\`a, and A.~Zeriahi 
for useful discussions and informations.\par\smallskip

\noindent The third-named author is partially supported by the project MTM2015-63699-P (Spanish MINECO and FEDER funds).

\bigskip


\smallskip

{\footnotesize
Daniel Li \\ 
Univ. Artois, Laboratoire de Math\'ematiques de Lens (LML) EA~2462, \& F\'ed\'eration CNRS Nord-Pas-de-Calais FR~2956, 
Facult\'e Jean Perrin, Rue Jean Souvraz, S.P.\kern 1mm 18 
F-62\kern 1mm 300 LENS, FRANCE \\
daniel.li@euler.univ-artois.fr
\smallskip

Herv\'e Queff\'elec \\
Univ. Lille Nord de France, USTL,  
Laboratoire Paul Painlev\'e U.M.R. CNRS 8524 \& F\'ed\'eration CNRS Nord-Pas-de-Calais FR~2956 
F-59\kern 1mm 655 VILLENEUVE D'ASCQ Cedex, FRANCE \\
Herve.Queffelec@univ-lille1.fr
\smallskip
 
Luis Rodr{\'\i}guez-Piazza \\
Universidad de Sevilla, Facultad de Matem\'aticas, Departamento de An\'alisis Matem\'atico \& IMUS,  
Apartado de Correos 1160 
41\kern 1mm 080 SEVILLA, SPAIN \\
piazza@us.es
}

\end{document}